\documentclass{amsart}
\usepackage{amssymb} %maths
\usepackage{amsmath} %maths
\usepackage{amsthm}
\usepackage{tikz}
\usepackage[colorlinks, linkcolor=blue,  citecolor=blue, urlcolor=blue]{hyperref}%

\usepackage{url}
\usepackage{graphicx}
\usepackage{float}
\usepackage{subcaption}

\graphicspath{{QC_PDE_figures/}}
\theoremstyle{Proposition}
\newtheorem{theorem}{Theorem}[section]
\newtheorem{corollary}[theorem]{Corollary}
\newtheorem{lemma}[theorem]{Lemma}
\newtheorem{proposition}[theorem]{Proposition}

\theoremstyle{definition}
\newtheorem{definition}[theorem]{Definition}

\theoremstyle{remark}
\newtheorem{remark}[theorem]{Remark}
\newtheorem{example}[theorem]{Example}

\newcommand{\norm}[1]{\Vert#1\Vert}

\newcommand{\abs}[1]{\vert#1\vert}

\DeclareMathOperator{\grad}{\nabla}
\DeclareMathOperator{\argmin}{{argmin}}

\DeclareMathOperator{\dist}{{dist}}
\newcommand{\bq}{\begin{equation}}
\newcommand{\eq}{\end{equation}}

\newcommand{\R}{\mathbb{R}}

\newcommand{\Rn}{\R^n}
\newcommand{\e}{\epsilon}
\newcommand{\bO}{\mathcal{O}}

\newcommand{\LQC}{\lambda_{QC}}

\newcommand{\Dir}{\mathcal{D}}

\begin{document}

\title[quasiconvex envelope]
{A partial differential equation for the strictly quasiconvex envelope}

\author{Bilal Abbasi \and Adam M. Oberman}
\address{Adam M. Oberman
\hfill\break\indent
Department of Mathematics and Statistics
\hfill\break\indent
McGill University 
\hfill\break\indent
{\tt adam.oberman@mcgill.ca}}

\date{\today}
%\keywords{Dirichlet boundary conditions, dynamic programming principle, $p$-Laplacian,
 %stochastic games, strong comparison principle, strong maximum principle,
 %two-player zero-sum games, uniqueness, viscosity solutions.}
%\subjclass[2000]{35B50, 35J25, 35J70, 49N70, 91A15, 91A24}
%\indent 2000{\it Mathematics Subject Classification.} 35J20, 35J60, 35J70.}

\begin{abstract}
In a series of papers Barron, Goebel, and Jensen studied Partial Differential Equations (PDE)s for quasiconvex (QC) functions \cite{barron2012functions, barron2012quasiconvex,barron2013quasiconvex,barron2013uniqueness}.  To overcome the lack of uniqueness for the QC PDE, they introduced a regularization: a PDE for $\e$-robust QC functions, which is well-posed.  Building on this work, we introduce a stronger regularization which is amenable to numerical approximation. We build convergent finite difference approximations, comparing the QC envelope and the two regularization.  Solutions of this PDE are strictly convex, and smoother than the robust-QC functions.
 \end{abstract}
\maketitle

\tableofcontents

%%%%%%%%%%%%%%%%%%%%%%%%%%%%%%%%%%%%%%%%%%%%%%%%
%%%%%%%%%%%%%%%%%%%%%%%%%%%%%%%%%%%%%%%%%%%%%%%%
\section{Introduction}
%%%%%%%%%%%%%%%%%%%%%%%%%%%%%%%%%%%%%%%%%%%%%%%%
%%%%%%%%%%%%%%%%%%%%%%%%%%%%%%%%%%%%%%%%%%%%%%%%

In a series of papers from about four years ago, Barron, Goebel, and Jensen introduced and studied partial differential equations (PDEs) for quasiconvexity \cite{barron2012functions, barron2012quasiconvex,barron2013quasiconvex,barron2013uniqueness}.  In this context, quasiconvexity means that the sublevel sets of a function are convex.  The study of convexity of level sets for obstacle problems has a long history, which includes \cite{caffarelli1982convexity} and \cite{kawohl1985rearrangements}, see also the more recent work \cite{colesanti2003quasi} and the references therein.   Quasiconvex (QC) functions appear naturally in optimization, since they generalize convex functions, yet still have a unique minimizer.  The property also appears in economics~\cite{avriel1988generalized}.  Earlier work by one of the authors studied a PDE for the convex envelope \cite{ObermanConvexEnvelope} which led to numerical method for convex envelopes \cite{ObermanEigenvalues, ObermanCENumerics}. 

Quasiconvexity is challenging because, unlike convexity, it is a nonlocal property (at least for functions which have flat parts).  This means that, even using viscosity solutions, there is no local characterization for quasiconvexity.   On the other hand, by using the more stable notion of robust quasiconvexity, Barron, Goebel and Jensen showed that these functions are characterized in the viscosity sense \cite{crandall1992user} by a partial differential equation \cite{barron2013quasiconvex, barron2013uniqueness}.

One motivation for this work was to build numerical solvers for the QC envelope PDE.   However, we had difficulties with both the QC and the robust-QC operators: the former lacks uniqueness, and the latter uses an operator defined over small slices of angles, (see the illustration Figure~\ref{fig:differentQC} below),  which leads to poor accuracy when using wide stencil finite difference schemes.   An alternative, presented in \cite{barron2013quasiconvex} was to use first order nonlocal PDEs solvers.   In a companion paper \cite{AbbasiLineSolverQuasiConvex}, we built a non-local  solver for the QC and for the robust-QC envelope \cite{AbbasiLineSolverQuasiConvex}.  This problem can be solved explicitly, and implemented efficiently.  By iteratively solving for the envelopes on lines in multiple directions, we approximated the solution of the problem in higher dimensions.   However, we are still interested in the PDE approach, which has advantages which come from a local expression for the operator. 

% However, the PDE approach has advantages over non-local solvers.    For example, the PDE conditions for convexity informed approaches to optimization problems with convexity constraints \cite{oberman2013numerical} and new PDE problems involving convexity \cite{cardaliaguet2009double}.

In this article, we build on the results of \cite{barron2012functions, barron2012quasiconvex, barron2013uniqueness} to obtain a PDE for strictly  quasiconvex (QC) functions. Strict quasiconvexity implies robust quasiconvexity.  Following the argument in \cite{barron2013uniqueness}, we establish uniqueness of viscosity solutions for the PDE.   Moreover, this operator is defined for all direction vectors, which makes it amenable to discretization using wide stencil schemes. 

 We consider the obstacle problem for the $\e$-strictly-convex envelope.   As is the case for robust-QC, we recover the QC-envelope as the regularization parameter $\e \to 0$.   While robust-QC functions can have corners in one dimension, strictly QC functions are smoother, see see Figure \ref{fig:differentQC} below, and the explicit formula for the solution in one dimension in \S\ref{sec:1dsoln}.

 We also build and implement convergent elliptic finite difference schemes \cite{BS91, ObermanSINUM} for the envelopes.  These are wide-stencil finite difference schemes, which can be developed using ideas similar to \cite{ObermanEigenvalues, ObermanMC,ObermanCENumerics}. Solutions to these PDEs can be found using an iterative method which is equivalent to the explicit Euler discretization of the parabolic equation~\cite{ObermanSINUM}. However the method has a nonlinear CFL condition which restricts the step size.  We find that alternating the line solver with several iterations of the parabolic PDE solver  significantly improves the speed of the solution.  Numerical solutions show that we obtain very similar results to the line solver for the QCE with small $\e = h^2/2$, where $h$ is the grid resolution.  We also compare large $\e$ solutions with comparable robust QCE, and find that solutions are smoother.  Formally, we show that solutions are $\e$ uniformly convex. 

The QC operator in two dimensions recovers the level set curvature operator.   We show that our discretization of our operator agrees with the Kohn-Serfaty \cite{kohn2007second} first order representation of the mean curvature operator in two dimensions.  See \S~\ref{sec:first_order}.

%%%%%%%%%%%%%%%%%%%%%%%%%%%%%%%%%%%%%%%%%%%%%%%%%

%%%%%%%%%%%%%%%%%%%%%%%%%%%%%%%%%%%%
\subsection{Convexity of level sets of a function}
%%%%%%%%%%%%%%%%%%%%%%%%%%%%%%%%%%%%
We give a brief informal derivation of the operator.  Given a smooth function $u: \Rn \to \R$, the direction of the gradient at $x_0$, $p = \grad u(x_0)$, is the normal to the sublevel set $\{ x \in \Rn  \mid u(x) \leq u(x_0) \}$.   The curvatures of the level set at $x_0$ are proportional to the eigenvalues of the Hessian of $u$, $M = D^2u(x_0)$, projected onto the tangent hyperplane at $x_0$, $P = P_{x_0} = \{ v \in \Rn \mid v\cdot p = 0\}$.   These curvatures are all positive if $v^\intercal M v \geq 0$ for all $v \in P$.  Thus, formally, the condition of local convexity of the level set is nonnegativity of the operator
\bq\label{QCoperator}
L_0(p,M) \equiv \min_{\abs{v} = 1} \left \{ v^\intercal M v \mid v\cdot p = 0   \right \}
\eq
This is the operator considered in \cite{barron2013quasiconvex} to study quasiconvex functions. 
%
%Defining the operator $a(v,p) = v \cdot p$, 
%we can write
%\[
%L^0(p,M) = \min_{\abs{v} = 1} \left \{ v^\intercal M v \mid a(v,p) = 0  \right \}
%\]
However, for technical reasons discussed below, they chose to relax the constraint $v \cdot p =0 $ to an inequality constraint $\abs{v \cdot p} \leq \e$, resulting in the operator
\[
L_\e(p,M) \equiv \min_{\abs{v} = 1} \left \{ v^\intercal M v \mid  \abs{v \cdot  p}\le\e   \right \},
\]
Our operator is obtained  by instead replacing the hard constraint $\abs{v\cdot p} \leq \e$ with a penalty in the objective function. 
So we define:
\[
F^\e(p,M) \equiv \min_{\abs{v} = 1}\left \{  v^\intercal M v + \frac{1}{{\e}} \abs{v^\intercal p} \right \},
\]

This choice of penalty gives the operator for uniformly convex level sets, as we show below. 

\autoref{fig:differentQC} illustrates the important differences in the notions of quasiconvexity discussed so far. Notice that the $\e$-robustly quasiconvex envelope can have an interval of global minimums, whereas our solution has a unique (global) minimum.

\begin{figure}[t]
    \centering
    \begin{subfigure}[b]{.5\textwidth}
    \centering
        \includegraphics[width=.95\textwidth]{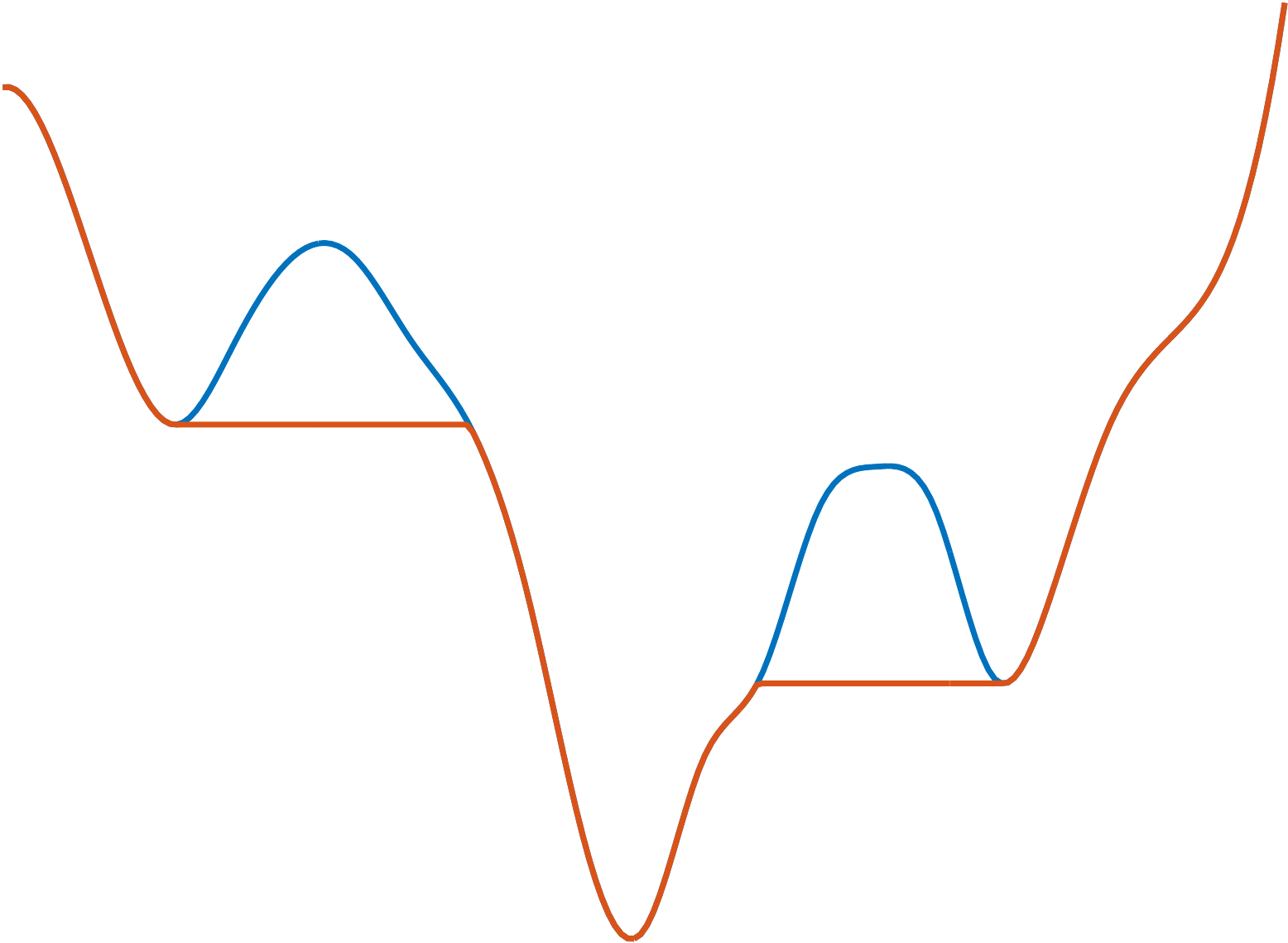}
        \caption*{quasiconvex envelope}
    \end{subfigure}
    ~
    \begin{subfigure}[b]{.5\textwidth}
    \centering
        \includegraphics[width=.95\textwidth]{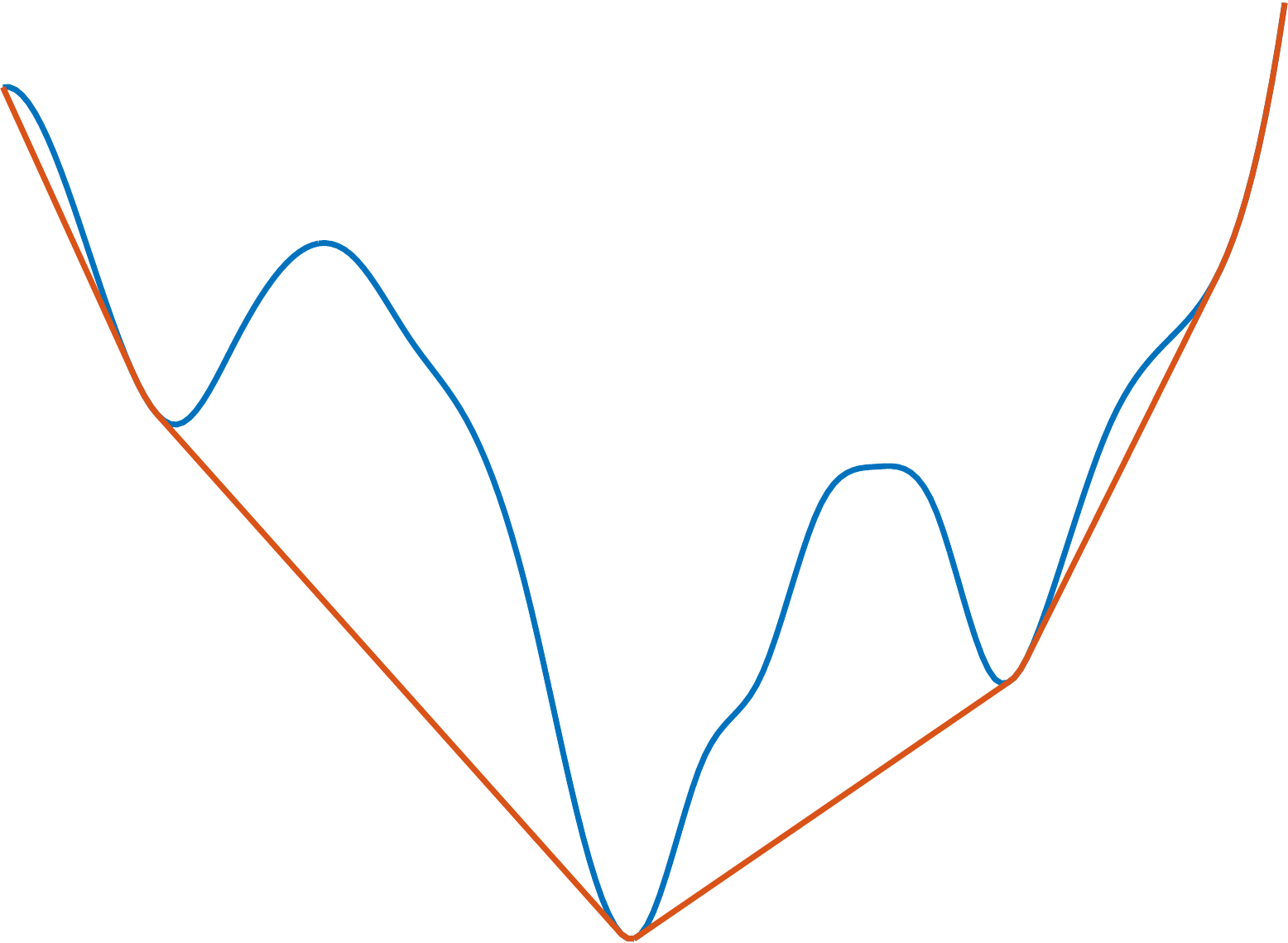}
        \caption*{convex envelope}
    \end{subfigure}
    \caption{Comparison of a function (blue) and its envelopes (red).}
    \label{fig:CEandQC1d}
\end{figure}

\begin{figure}[t]
~%%constraint set of original QC operator
\begin{minipage}{0.32\textwidth}
\centering
\begin{tikzpicture}
\draw (1.414,1.414) arc [very thick,radius=2, start angle=45, end angle= 135];
\filldraw (0,2) circle(1pt);
\draw [red,thick,->] (0,2) -- (0,2.95);
\draw [blue,thick,->] (0,2) -- (.75,2);
\draw [blue,thick,->] (0,2) -- (-.75,2);
\draw[white](-1,0.99)--(1,0.99);
\end{tikzpicture}
\caption*{$\nabla u(x)^\intercal v = 0$}
\end{minipage}
~%%constraint set of alpha robust operator
\begin{minipage}{0.32\textwidth}
\centering
\begin{tikzpicture}
\draw (1.414,1.414) arc [ultra thick,radius=2, start angle=45, end angle= 135];
\draw [red,->,thick] (0,2) -- (0,2.95);
  \fill[fill=blue!19!white,opacity=0.5]
    (0,2) -- (.6495,1.625) arc (-30:30:.75);
  \fill[fill=blue!19!white,opacity=0.5]
    (0,2) -- (-.6495,2.375) arc (150:210:.75);
\draw[blue] (0,2)--(0.6495,1.625);
\draw[blue] (0,2)--(0.6495,2.375);
\draw[thick,blue] (0.6495,1.625) arc[radius=.75,start angle=-30, end angle=30];
\draw[blue] (0,2)--(-0.6495,1.625);
\draw[blue] (0,2)--(-0.6495,2.375);
\draw[thick,blue] (-0.6495,2.375) arc[radius=.75,start angle=150, end angle=210];
\draw (0,2) circle(1pt);
\draw[white](-1,0.99)--(1,0.99);
\end{tikzpicture}
\caption*{$\abs{\nabla u(x)^\intercal v} \le \e$}
\end{minipage}
~%%constraint set of our operator
\begin{minipage}{0.32\textwidth}
\centering
\begin{tikzpicture}
\draw (1.414,1.414) arc [ultra thick,radius=2, start angle=45, end angle= 135];
\draw[color=blue!, fill=blue!19, opacity=0.5](0,2) circle (.75);
\draw[thick, color=blue] (0,2)circle(.75);
\draw [red,->,thick] (0,2) -- (0,2.95);
\draw (0,2) circle(1pt);
\draw[white](-1,0.99)--(1,0.99);
\end{tikzpicture}
\caption*{$\abs{v}=1$}
\end{minipage}
\caption{Constraint sets of different PDEs in two dimensions. Red represents the gradient, blue (shaded) represents the feasible vectors, and black represents the level set.}
\label{constraintset}
\end{figure}

\begin{figure}[t]
    \centering
    \begin{subfigure}[b]{.5\textwidth}
    \centering
        \includegraphics[width=.95\textwidth]{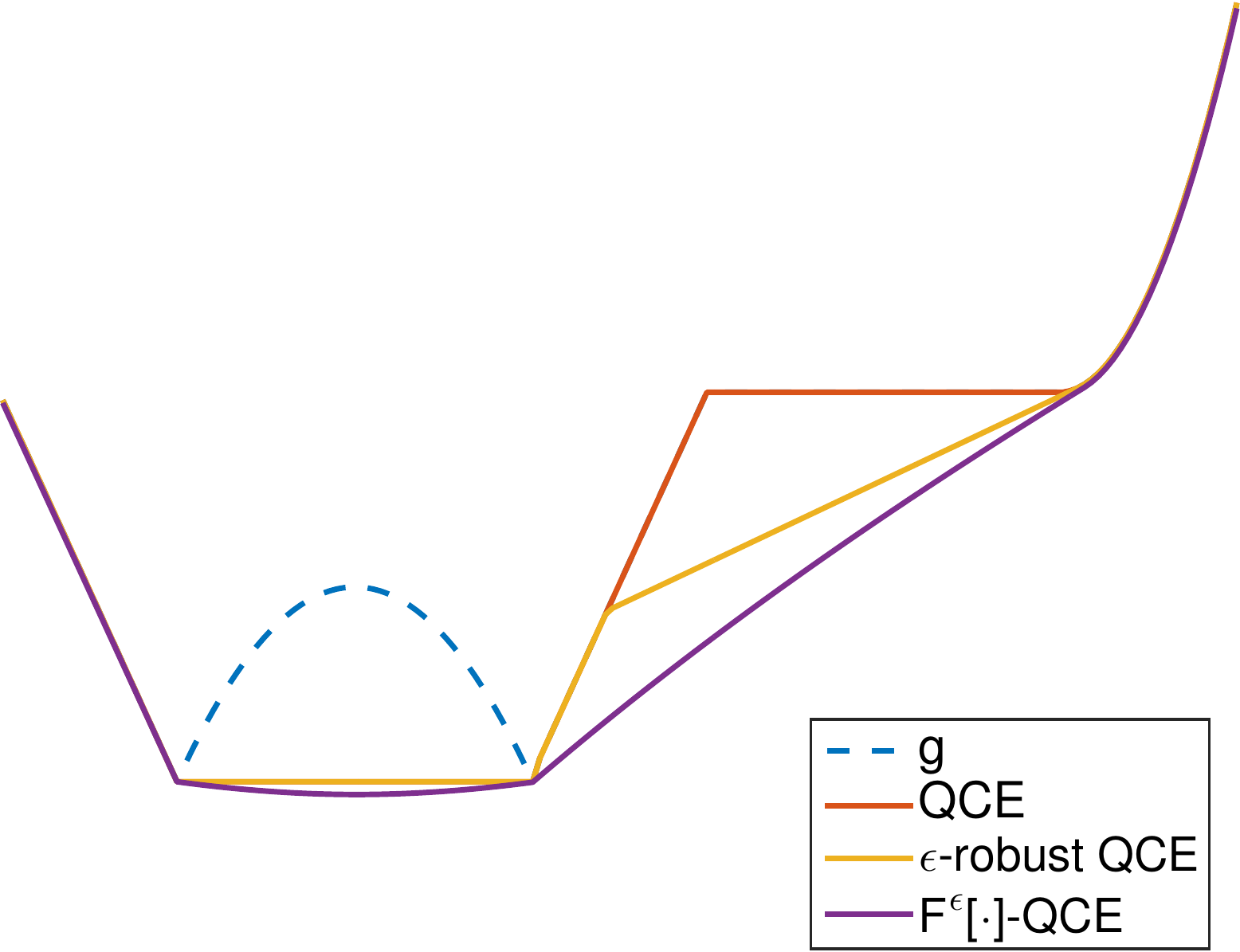}
        \caption{Comparison of the envelopes derived from different notions of quasiconvexity.}
    \end{subfigure}
    ~
    \begin{subfigure}[b]{.5\textwidth}
    \centering
        \includegraphics[width=.95\textwidth]{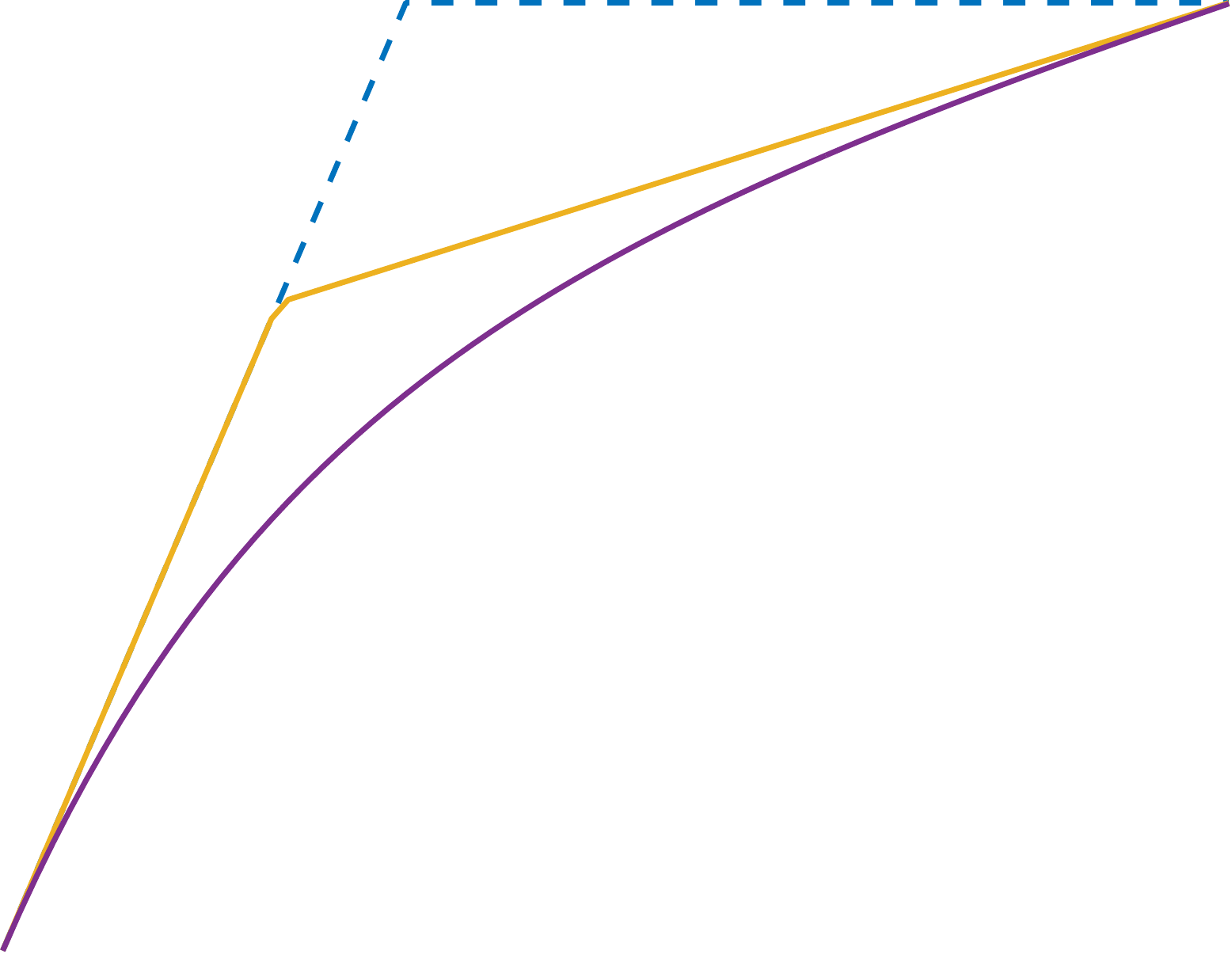}
        \caption{Closer look at the difference between $\e$-robust quasiconvexity and strict quasiconvexity.}
    \end{subfigure}
    \caption{Comparison between solutions of different operators.}
    \label{fig:differentQC}
\end{figure}

%%%%%%%%%%%%%%%%%%%%%%%%%%%%%%%%%%%%
\subsection{Basic definitions}
We recall some basic definitions and establish our notation.  For a reference, see~\cite{BoydBook}.
The set $S$ is \emph{convex} if whenever $x$ and $y$ are in $S$ then so is the line segment $[x,y]\equiv\{tx + (1-t)y$: $t \in [0,1]\}$. 
We say that the function $u: \Rn \to \R$ is convex if 
\begin{equation}
\label{convexity_u}
u(tx + (1-t)y) \le t u(x) + (1-t) u(y), 
\quad \text{ for every } x,y \in \Rn,  0 \le t \le 1. 
\end{equation}
The \emph{convex envelope} of a function $g$, hereby denoted $g^{CE}$, is the largest convex function majorized by $g$:
\bq\label{CE}
g^{CE}(x) \equiv \sup \{ v(x)\mid\text{$v$ is convex and $v\le g$}  \}. 
\eq
In terms of sets, if $S$ is given by the sublevel set of a function $u$:
\bq\label{levelset}
S = S_\alpha(u)\equiv\{x \in \Rn \mid u(x)\leq\alpha\}
\eq
then convexity of $S$ is equivalent to the following condition 
\[
u(x), u(y) \le \alpha \implies u(tx + (1-t)y) \le \alpha,
\quad \text{ for every } 0 \le t \le 1. 
\]
This condition, when applied to every level set $\alpha\in\R$, characterizes quasiconvexity of the function $u$. That is, $u$ is \emph{quasiconvex} if every sublevel set $S_\alpha(u)$ is convex.  Equivalently, $u$ is quasiconvex if 
\bq
\label{QCcond0}
u(tx + (1-t)y) \le \max \{ u(x), u(y) \}, 
\quad \text{ for every } x,y \in \Rn,  0 \le t \le 1. 	
\eq
%It follows from this definition that any convex function is quasiconvex. 

Given $g:\Rn \to \R$, the \emph{quasiconvex envelope} of $g$ is given similarly by
\begin{equation}
	\label{QCE}
g^{QCE}(x) \equiv \sup \{ v(x)\mid\text{$v$ is quasiconvex and $v\le g$}  \} 
\end{equation}
\begin{remark}
Since the maximum of any two quasiconvex (convex) functions is quasiconvex (convex) as well, the supremums in \eqref{CE} and \eqref{QCE} are well-defined. 	
\end{remark}

Figure (\ref{fig:CEandQC1d}) provides a visual comparison between the convex envelope and the quasiconvex envelope in one dimension.

%
%\begin{remark}[Technical conditions]
%In definitions \eqref{CE} and \eqref{QCE} we assume that $g$ is continuous and satisfies the following coercivity condition:
%\[
%g(x)\to\infty\quad\text{as $\abs{x}\to\infty$}.
%\]
%The growth condition at infinity guarantees that the quasiconvex envelope is bounded below. On a bounded domain, we simulate this growth condition by requiring that the function satisfy the condition $\partial u / \partial n =\infty $ where $n$ denotes the outward normal of $\Omega $ \cite{alvarez1997convex}.	
%\end{remark}

%%%%%%%%%%%%%%%%%%%%%%%%%%%%%%%%%%%%%%%%%%%%%%%%
%%%%%%%%%%%%%%%%%%%%%%%%%%%%%%%%%%%%%%%%%%%%%%%%
\subsection{Viscosity solutions} \label{sec:viscositysolns}
%%%%%%%%%%%%%%%%%%%%%%%%%%%%%%%%%%%%%%%%%%%%%%%%
%%%%%%%%%%%%%%%%%%%%%%%%%%%%%%%%%%%%%%%%%%%%%%%%
Suppose that $\Omega\subset\Rn$ is a domain. Let $S^n$ be the set of real symmetric $n\times n$ matrices, and take $N\le M$ to denote the usual partial ordering on $S^n$; namely that $N-M$ is negative semi-definite.

\begin{definition}\label{def:elliptic}
The operator $F(x,r,p,M):\Omega\times\R\times\Rn\times S^{n} \to\R$ is \emph{degenerate elliptic} if
\[
F(x,r,p,M)\le F(x,s,p,N)\quad\text{whenever}\ r\le s\ \text{and}\ N\le M
\]
\end{definition}

\begin{remark}
For brevity we use the notation $F[u](x)\equiv F(x,u(x),\nabla u(x), D^2u(x))$. 
%When it is clear from the context, we simply write $F[u]$ instead of $F[u](x)$. 
%For example, writing $F[u]=0$ in $\Omega$ is equivalent to writing $F[u](x)=0$ for every $x\in\Omega$.
\end{remark}

\begin{definition}[Viscosity solutions]
\label{def:viscosity}
We say the upper semi-continuous (lower semi-continuous) function $u:\Omega\to\R$ is a \emph{viscosity subsolution (supersolution)} of $F[u]=0$ in $\Omega$ if for every $\phi\in C^2(\Omega)$, whenever $u-\phi$ has a strict local maximum (minimum) at $x\in\Omega$
\[
F(x,u(x),\nabla\phi(x),D^2\phi(x))\le0\ (\ge0)
\]
Moreover, we say $u$ is a \emph{viscosity solution} of $F[u]=0$ if $u$ is both a viscosity sub- and supersolution.
\end{definition}

%\begin{remark}
%We use the shorthand expression "$\phi$ touches $u$ from above (below)" to express the above condition. 
%\end{remark}

%\begin{proposition}[Ordering of subsolutions]
%Suppose $u$ is a strict viscosity subsolution of $-\LQC[u]=0$. Then $u$ is also a subsolution of $\e-\LQC^{\e^2}[u]=0$ and $\
%\end{proposition}

%%%%%%%%%%%%%%%%%%%%%%%%%%%%%%%%%%
\subsection{The quasiconvex envelope}
%%%%%%%%%%%%%%%%%%%%%%%%%%%%%%%%%%
Use the notation $\LQC[u](x) \equiv L_0(\grad u(x),D^2u(x))$.
The obstacle problem for the QC envelope is given by
\bq\label{QCobstacle}
\tag{Ob}
\begin{cases}
\max\{u-g,-\LQC[u](x)\}=0 & x\in\Omega\\
u =g & x\in\partial\Omega
\end{cases}
\eq
This PDE can have multiple viscosity solutions for a given $g$ \cite{barron2013quasiconvex}. However, there is a unique \emph{quasiconvex} solution which is the QCE of $g$.  

Failure of uniqueness  in general can be seen in the following counter-example, which is similar to~\cite[Example 3.1]{barron2013quasiconvex}.
\begin{example}
Consider $\Omega=\{x^2+y^2<1\}, u(x,y) = -y^4, v(x,y) = -(1-x^2)^2$, so that $u=v$ on $\partial\Omega$ . Now consider $g\equiv u$ on $\Omega$.  We can calculate that $\LQC[u]=\LQC[v]=0$. In summary, $u$ and $v$ are both solutions of \eqref{QCobstacle} but $u\neq v$ in $\Omega$ (in fact $v<u$ in $\Omega$). This contradicts the uniqueness of \eqref{QCobstacle}.
\end{example}

It is shown in \cite{barron2013quasiconvex} that continuous quasiconvex functions necessarily satisfy $\LQC[u] \geq 0$, in the viscosity sense. The converse, however, is not true; consider the function $u(x)=-x^4$, which is concave yet satisfies $\LQC[u]=0$ at $x=0$.
In fact, using this as a sufficient condition is only possible when $u$ has no local maxima \cite{barron2013quasiconvex}.

This operator can be used to completely characterize the set of continuous $\e$-robustly quasiconvex functions, defined below.
\begin{definition}[$\epsilon$-robustly quasiconvex]

The function  $u:\Rn\to\R$ is \emph{$\epsilon$-robustly quasiconvex} if $u(x) +y^\intercal x$ is quasiconvex for every $\abs{y}\leq\epsilon$
\end{definition}
In particular, $\epsilon$-robustly quasiconvex functions are functions whose quasiconvexity is maintained under small linear perturbations. 
Write $\LQC^\epsilon[u](x) \equiv L_\epsilon(\grad u(x),D^2u(x))$.
\begin{proposition}[Characterization of $\epsilon$-robustly quasiconvex functions \cite{barron2013quasiconvex}]
\label{prop:robustQC}
The upper semicontinuous function, $u:\Omega\to\R$, is $\e$-robustly quasiconvex if and only if $u$ is a viscosity subsolution of $\LQC^\epsilon[u]=0$.
\end{proposition}

Next we show that viscosity subsolutions (defined below) of our operator  are robustly-QC (which also implies they are QC).  This allows us  avoid a technical argument from \cite{barron2013quasiconvex}.  We believe that stronger results hold: see the formal analysis in \S\ref{sec:Formal}.

\begin{proposition}\label{prop:subQC}
Suppose  $u\in USC(\overline\Omega)$ is a viscosity subsolution of $\e-F^\e[u]=0$. Then $u$ is $\e^2$-robustly quasiconvex.
\end{proposition}
\begin{proof}
First observe that for any $\phi\in C^2$ we have the following inequality:

\begin{align*}
F^\e[\phi](x) - \frac{\alpha}{\e} &=  \min_{\abs{v}=1}\left\{\frac{1}{{\e}}\abs{\nabla \phi(x)\cdot v} + \phi_{vv}(x)  - \frac{\alpha}{\e} \right\}\\
&\leq   \min_{\{\abs{v}=1,\abs{\nabla \phi(x)\cdot v}\le \alpha \}}\left\{ \phi_{vv}(x)\right\}
= \lambda_{QC}^{\alpha}[\phi](x)
\end{align*}
Choosing $\alpha=\e^2$ we see that any viscosity subsolution of $\e-F^\e[u]= 0$ is also a viscosity subsolution of $-\LQC^{\e^2}[u]= 0$. Thus  $u$ is $\e^2$-robustly quasiconvex. 
\end{proof}

%%%%%%%%%%%%%%%%%%%%%%%%%%%%%%%%%%%%%%%%%%%%%%%%
%%%%%%%%%%%%%%%%%%%%%%%%%%%%%%%%%%%%%%%%%%%%%%%%
\section{Properties of solutions}
%%%%%%%%%%%%%%%%%%%%%%%%%%%%%%%%%%%%%%%%%%%%%%%%
%%%%%%%%%%%%%%%%%%%%%%%%%%%%%%%%%%%%%%%%%%%%%%%%

In this section we present technical arguments proving the uniqueness of solutions of our PDEs, and discuss some relevant properties.

%%%%%%%%%%%%%%%%%%%%%%%%%%%%%%%%%%%%%%%%%%%%%%%%
\subsection{Comparison principle}
%%%%%%%%%%%%%%%%%%%%%%%%%%%%%%%%%%%%%%%%%%%%%%%%
In this section we will show that a weak comparison principle holds for the Dirichlet problem of $h-F^\e[u]=0$, for $h>0$ and $\e > 0$. Comparison also holds for the corresponding obstacle problem. The proof we present is based on the uniqueness proof of $\LQC[u]=g$, for $g>0$, presented in \cite{barron2013uniqueness}.  The result is simpler because our operator is continuous as a function $(p,M)\mapsto F^\e(p,M)$ for $\e > 0$. 

Write
\bq\label{Feps}
F^\e[u](x)=F^\e(\nabla u(x),D^2u(x)) \equiv \min_{\abs{v}=1}\left\{\frac{1}{{\e}}\abs{\nabla u(x)\cdot v} + v^\intercal D^2(x)v\right\}
\eq
Note that $-F^\e(p,M)$ is elliptic by Definition \ref{def:elliptic}. 
We consider the following PDEs:

\bq 
\label{eQC}
\tag{$\e$-QC}
h(x)-F^\e[u](x) = 0 \quad x\in\Omega
\eq
\bq 
\label{eQCE}
\tag{$\e$-QCE}
\max\{u(x)-g(x),h(x)-F^\e[u](x)\} = 0 \quad x\in\Omega
\eq
where $\Omega\subset\Rn$ is an open, bounded, and convex domain, and $g:\overline\Omega\to\R$ is continuous. In the latter equation, $g$ is the obstacle. We also impose the following condition on $h$:
\bq\label{h_assumption}
\text{$h:\overline\Omega\to\R$ is continuous and positive.}
\eq
Enforce the following Dirichlet boundary data:
\bq\label{dirichlet}
\tag{Dir}
u(x) = g(x)\quad x\in\partial\Omega
\eq

\begin{remark}[Continuity up to the boundary]\label{rmk:bc}
In general, viscosity solutions of \eqref{eQCE} need not be continuous up to the boundary.
To apply \cite{BS91} for  convergence of numerical schemes, we need a strong comparison principle, which requires that solutions be continuous up to the boundary.  The following assumption ensures continuity up to the boundary
There is a convex domain $\Omega_L\supset\Omega$ and a continuous, quasiconvex function $g_0: \Omega_L \to \R$, with 
\[
g_0(x)\le g(x) \quad \text{ for } x \in \Omega,
\qquad
g_0(x)=g(x), \quad \text{ for } x\in \Omega_L\setminus\Omega.
\]
Continuity follows because we have $g_0 \leq u \leq g$ which gives $u = g$ on $\partial \Omega$.
An alternative to this condition is to prove convergence in a neighbourhood of the boundary.  The convergence proof in this setting can be found in~\cite{froeseGauss}.
 \end{remark}

Next we state a technical, but standard, viscosity solutions result, which gives the comparison principle in the case where we have strict sub and supersolutions.

\begin{proposition}
	[Comparison principle for strict subsolutions \cite{CIL}]\label{prop:Comp1}
Consider the Dirichlet problem  for the degenerate elliptic operator $F(p,M)$ on the bounded domain $\Omega$.  Let $u \in USC(\bar \Omega)$ be a viscosity subsolution and let $v \in LSC(\bar \Omega)$ be a viscosity supersolution.  Suppose further that for $\sigma > 0$,
\begin{align*}
	F[u] + \sigma  &\le  0 \text{ in } \Omega
\\
 	F[v] &\ge 0 \text{ in } \Omega
\end{align*}
holds in the viscosity sense.  Then the comparison principle holds:
\[
\text{ if $u\leq v$ on $\partial\Omega$ then $u\le v$ in $\Omega$ }
\]
\end{proposition}

\begin{remark}
In \cite[Section 5.C]{CIL}, it is explained how the main comparison theorem, \cite[Theorem 3.3]{CIL},
can be applied when it is possible to perturb a subsolution to a strict subsolution.   This version of the theorem is what we state in Proposition \ref{prop:Comp1}.  This result was used in \cite[Theorem 3.1]{Bardi2006709} and  \cite{bardi2013comparison} to prove a comparison principle.
\end{remark}

Using the the same perturbation technique from \cite{barron2013uniqueness}, and consequently applying Proposition \ref{prop:Comp1}, we obtain the following comparison result.

\begin{proposition}[Comparison principle]\label{comparison}

Consider the Dirichlet problem given by \eqref{eQC}, \eqref{dirichlet} and assume \eqref{h_assumption} holds. Let $u\in USC(\overline\Omega)$ be a viscosity subsolution and $v\in LSC(\overline\Omega)$ be a viscosity supersolution of \eqref{eQC}.  Then the comparison principle holds:
\[
\text{ if $u\leq v$ on $\partial\Omega$ then $u\le v$ in $\Omega$ }
\]
\end{proposition}

\begin{proof}
We will show that we can perturb $u$ to a function $u^\sigma$ satisfying $u^\sigma\le v$ on $\partial\Omega$ and
\[
h-F^\e[u^\sigma]<0 \quad\text{in $\Omega$}
\]
in the viscosity sense. Applying Proposition \ref{prop:Comp1}, we will have that $u^\sigma\le v$ in $\Omega$. Taking $\sigma\to0$ yields the desired result. 

Fix $\sigma>0$ and define the following perturbation of $u$, and notice that for $x\in\partial\Omega$ we have the following relation:
\[
u^\sigma(x) \equiv u(x) - \sigma \left(\max_{y\in\partial\Omega}u(y)-u(x)\right)\le u(x) - \sigma (u(x)-u(x)) = u(x)\le v(x)
\]
That is, $u^\sigma\le v$ on $\partial\Omega$. Next, because $u$ is a subsolution, we have the following
\[
h-\frac{1}{1+\sigma}F^\e[u^\sigma]  \le 0
\]
leading to
\[
h - F^\e[u^\sigma]\le -\sigma h \le -\sigma h_{\min} <0\quad\text{where  $h_{\min}\equiv \min_{x\in\overline\Omega}h(x)$}
\qedhere
\]
\end{proof}

We use this result to prove a weak comparison principle for the obstacle problem given by \eqref{eQCE}, \eqref{dirichlet}.

\begin{corollary}[Comparison principle for the obstacle problem]
\label{e-CP} 
Consider the Dirichlet problem given by \eqref{eQCE}, \eqref{dirichlet} and assume \eqref{h_assumption} holds. Let $u\in USC(\overline\Omega)$ be a viscosity subsolution and $v\in LSC(\overline\Omega)$ be a viscosity supersolution of \eqref{eQCE}.  Then the comparison principle holds:
\[
\text{ if $u\leq v$ on $\partial\Omega$ then $u\le v$ in $\Omega$ }
\]
\end{corollary}
\begin{proof}
We begin by considering the domain $\Omega_g\equiv\{x\in\Omega:v(x)<g(x)\}$. Then $h-F^\e[v]\ge 0$ in $\Omega_g$ and $v=g$ on $\partial\Omega_g$. That is, $v$ is a viscosity supersolution of $h-F^\e[v]=0$ in $\Omega_g$. Now, by the definition of viscosity subsolutions we have that $u\le g$ and $h-F^\e[u]\le 0$ in $\Omega$, and thus also $\Omega_g$. This allows us to conclude that $u\le v$ on $\partial\Omega_g$. Therefore by Proposition \ref{comparison}, $u\le v$ in $\Omega_g$. Concluding, in $\Omega\setminus\Omega_g$ we necessarily have that $v\ge g\ge u$.
\end{proof}

\subsection{The strictly quasiconvex envelope}\label{sec:approxOb}
%%%%%%%%%%%%%%%%%%%%%%%%%%%%%%%%%%%%%%%%%%%%%%%%

We formulate the \emph{strictly quasiconvex envelope} of a function $g$ as the unique viscosity solution of the following obstacle problem.
\bq 
\label{approxOb}
\tag{$\e$-Ob}
\begin{cases}
\max\{u(x)-g(x),\e-F^\e[u](x)\} = 0 & x\in\Omega\\
u(x) = g(x) & x\in\partial\Omega
\end{cases}
\eq

\begin{remark}[Convergence of approximate solutions]\label{rmk:QCstability}
It is clear that as $\e\to0$, the penalization term in $F^\e[u]$ tends to infinity. The result is that $F^\e[u]\to\LQC[u]$ as $\e\to0$. From this observation, the standard stability result of viscosity solutions, and the quasiconvexity of subsolutions of $\e-F^\e[u]=0$ , one can then apply the same argument presented in the proof of \cite[Theorem 5.3]{barron2013quasiconvex} to conclude that the unique viscosity solutions of \eqref{eQCE}, \eqref{dirichlet} converge to the quasiconvex envelope as $\e\to0$. This result allows us to compute asymptotic approximations of the quasiconvex envelope of a given obstacle.
\end{remark}

%%%%%%%%%%%%%%%%%%%%%%%%%%%%%%%%%%%%%%%%%%%%%%%%%
\subsection{Solution formula in one dimension}\label{sec:1dsoln}
%%%%%%%%%%%%%%%%%%%%%%%%%%%%%%%%%%%%%%%%%%%%%%%%%
In one dimension, $\e-F^\e[u]$ is simply the Eikonal operator with a small diffusion term. When considering a solution $u$ of \eqref{approxOb}, whenever $u(x) < g(x)$, we have that $F^\e[u](x) = \e$, which gives 
\bq\label{F1d}
\e u_{xx} + \abs{u_x}=\e^2 %\quad\text{when $u<g$}.
\eq
Define $\Omega_g\equiv\{x:u(x) < g(x)\}$. Note that $\Omega_g$ need not be connected; it can be written as the union of finitely many intervals (refer to \autoref{fig:differentQC}).  However, we can solve the equation in each interval.  
\begin{lemma}[1D solution]
The viscosity solution of the one dimensional PDE for the operator \eqref{Feps}, \eqref{F1d},
along with boundary conditions $u(0)=0, u(W) = H$
is given by the following.  Set $S\equiv H/\e^2W$. 
Then
\[
u^\e(x) = 
\begin{cases}
\pm \e^2 x & S=\pm1\\
\e^2x + C^+(1-\exp(-x/\e)) & S>1\\
-\e^2x + C^-(1-\exp(x/\e)) & S<-1\\
\e^2\abs{x-x^*} + \e^3(\exp(-\abs{x-x^*}/\e)-1) + u_0 & \abs{S}<1
\end{cases}
\]
where
\[
C^\pm \equiv \frac{H\mp\e^2 W}{\pm\e(1-\exp(\mp W/\e))}
\]
and $u_0\in\R$ and $x^*\in I = (0,W)$ are constants.
\end{lemma}

\begin{proof}
If there exists $x_0$ in the interval $I$ such that $u_x(x_0)=0$, then \eqref{F1d} implies $u_{xx}(x_0)>0$. That is, if there exists an interior critical point, then $u$ must be strictly convex in $I$. This allows us to break down the analysis of \eqref{F1d}, restricted to each interval, into several cases: (i) $u_x<0$ in $I$, (ii) $u_x>0$ in $I$, and (iii) $u_x(x_0)=0$ for some $x_0 \in I$. 
In each case, we can solve a linear second order ODE.  The case where $S = \pm 1$ is degenerate: the solution is linear.   The case where $\abs{S} > 1$ is not difficult.  The final case, where $\abs{S} < 1$ corresponds to (iii).   

In this case, $u_x<0$ for $x<x_0$ and $u_x>0$ for $x>x_0$. Then 
\begin{equation}\label{u1dsoln}
u(x) = \abs{x-x_0} + (\exp(-\abs{x-x_0})-1) + u_0	
\end{equation}
For some $u_0\in\R$. Taylor expansion shows that $u(x) = x^2/2 + \bO(x^3)$ for $x$ near $x_0$. 
Finding $x_0$ (or $u_0$) analytically is infeasible. But we can argue that the solution is correct by a continuity argument. 
Given the function $u$ defined by \eqref{u1dsoln},  define the continuous function $S_1(y)\equiv\frac{u(y+W_1)-u(y)}{W_1}$. For small $W_1$ and $y\ll x_0$ we have that $S_1(y)<-1$ (by the earlier discussion). Similarly for $y\gg x_0$ , we have that $S_1(y)>1$. Therefore by the intermediate value theorem there exists $y^*$ such that $S_1(y^*)=S$.  This leads to the correct choice of constants in \eqref{u1dsoln} which achieve the boundary values. 
\end{proof}
%\begin{remark}
%The reason for the formulation of the solution in terms of $S$ is that the explicit dependence on the graph of the function is more apparent.
%\end{remark}

%%%%%%%%%%%%%%%%%%%%%%%%%%%%%%%%%%%%%%%%%%
\section{Formal analysis}\label{sec:Formal}
%%%%%%%%%%%%%%%%%%%%%%%%%%%%%%%%%%%%%%%%%%
In this section we make some formal arguments about quasiconvexity of  strict subsolutions of related PDEs. We present formal arguments justifying the strict and uniform convexity of the level sets of strict subsolutions of the operators. Additional efforts as in \cite{barron2013quasiconvex}, \cite{carlier2012} could  make the arguments rigorous. 

\begin{definition}[Strict and uniform convexity of sets]
Let $S\subset\Rn$ be a domain and suppose $x,y\in\partial S$. We say $S$ is \emph{strictly convex} if $(x,y)$ is in the interior of $S$. Moreover, define $\bar x\equiv(x+y)/2$, then we say $S$ is \emph{uniformly convex} if $\dist(\bar x,\partial S) = \bO(h^2)$ as  $h = |x-y| \to 0$.

%We say $S$ is \emph{strictly convex} if whenever $x,y\in\partial S$ then $(x,y)$ is in the interior of $S$.  The set  $S$ is  \emph{uniformly convex} if whenever $x,y \in \partial S$, then define $\bar x=(x+y)/2$.  Then we require $\dist(\bar x,\partial S) = \bO(h^2)$ as  $h = |x-y| \to 0$. \blue{more on this later}	
\end{definition}

\begin{definition}
\label{Defn:dirQC}
$u:\mathbb{R}^n\rightarrow\mathbb{R}$ is quasiconvex along a direction $v$ if for any $x\in\mathbb{R}^n$, the function $\tilde u(t)\equiv u(x+tv)$ is quasiconvex for $t\in\mathbb{R}$.
\end{definition}

We also appeal to the following proposition, which is elementary from the definition of quasiconvexity.
\begin{proposition}
\label{dirQC}
$u$ is quasiconvex if and only if $u$ is quasiconvex in every direction $v$. 
\end{proposition}

%%%%%%%%%%%%%%%%%%%%%%%%%%%%%%%%%%%%%%%%%%
\subsection{Convexity of the level sets of solutions}\label{sec:convexity}
%%%%%%%%%%%%%%%%%%%%%%%%%%%%%%%%%%%%%%%%%%

For the remainder of this section, it suffices to consider subsolutions of $\e - \LQC[u]=0$. This is because of the ordering of the operators:
\bq\label{eq:inequalities}
-\LQC[u]\ge-\LQC^{\e^2}[u]\ge \e - F^\e[u]
\eq
The first inequality follows naturally from the restriction of the constraint set. The second inequality is evident from the proof of Proposition \ref{prop:subQC}
\begin{proposition}[Strictly convex level sets of subsolutions]
\label{strictconvexity}
Suppose $u\in C^2(\overline\Omega)$ is a subsolution of $\e-\LQC[u]=0$. Then $u$ has strictly convex level sets.
\end{proposition}

\begin{proof}
%Recall that an open, bounded domain $S$ is strictly convex if whenever $x,y\in\partial S$ we have that $tx+(1-t)y\in S$ for every $t\in(0,1)$.

Using the notation given by \eqref{levelset}, fix $\alpha\in\R$ and consider $x,y\in\partial S^\alpha[u]$, that is, $u(x)=u(y)=\alpha$. We would like to show that $tx + (1-t)y$ is strictly in the interior of $S^\alpha[u]$ for $t\in(0,1)$. By \cite[Theorem 2.7]{barron2013quasiconvex} $u$ is quasiconvex, so this amounts to showing that $u(tx+(1-t)y)<\alpha$. 

Consequently, suppose for contradiction that there exists $t_z\in(0,1)$ such that $z\equiv t_z + (1-t_z)y$ satisfies $u(z) = \alpha$. Define $d\equiv x-y$ and consider the function $\tilde u(t) \equiv u(y+t d)$ for $t\in[0,1]$ which is quasiconvex by Proposition \ref{dirQC}. Then $\tilde u(0) =\tilde u(t_z) = \tilde u(1) = \alpha$, so that by Rolle's theorem there exists $0<t_1<t_z$ and $t_z<t_2<1$ such that $\tilde u'(t_1) = \tilde u'(t_2)=0$. Denoting $x_i = x+t_i d$ for $i=1,2$, this implies $\grad u(x_i)^\intercal d=0$. Next, because $u$ is a subsolution, we have that $\e - \LQC[u]\le0$. In particular:
\begin{align*}
0<\e&\le \min_{\abs{v}=1,\grad u(x)^\intercal v= 0} v^\intercal D^2(x_i)v 
\\
&\le \frac{d^\intercal D^2(x_i)d}{\abs{d}^2}
\end{align*}
That is, we have that $0<d^\intercal D^2(x_i)d$, so that $\tilde u''(t_i)>0$. Therefore $\tilde u$ has two distinct strict local minima. This contradicts the quasiconvexity of $\tilde u$.
\end{proof}

\begin{proposition}[Uniformly convex level sets of subsolutions]
Suppose $u\in C^2(\overline\Omega)$ is a subsolution of $\e-\LQC[u]=0$. Then $u$ has uniformly convex level sets.
\end{proposition}

\begin{proof}[Formal proof]
Suppose $x,y\in\Omega$ such that $u(x)=u(y)=\alpha$. Let $\bar x$ be the midpoint of line segment joining $x$ and $y$, $h = \abs{\bar x - y}$, and $d = (\bar x-y)/\abs{\bar x -y}$. We see that $\grad u(\bar x)^\intercal d= 0$, so that we have:
\begin{align*}
\e &\le \min_{\abs{v}=1,\grad u(\bar x)^\intercal v=0} v^\intercal D^2u(\bar x) v\\
&\le d^\intercal D^2u(\bar x) d\\
&= \frac{u(\bar x + hd) + u(\bar x -hd) - 2u(\bar x)}{h^2} +\bO(h^2)\\
&= \frac{2\alpha - 2u(\bar x)}{h^2} + \bO(h^2)
\end{align*}
Rearranging for $u(\bar x)$ yields the following inequality:
\[
u(\bar x) \le \alpha -\frac{\e h^2}{2} + \bO(h^4)
\qedhere
\]
\end{proof}

%%%%%%%%%%%%%%%%%%%%%%%%%%%%%%%%%%%%%%%%%%%%%%%%
%%%%%%%%%%%%%%%%%%%%%%%%%%%%%%%%%%%%%%%%%%%%%%%%
\subsection{Directional quasiconvexity}\label{sec:dirQC}
%%%%%%%%%%%%%%%%%%%%%%%%%%%%%%%%%%%%%%%%%%%%%%%%
%%%%%%%%%%%%%%%%%%%%%%%%%%%%%%%%%%%%%%%%%%%%%%%%

Proposition (\ref{dirQC}) provides a convenient characterization of quasiconvex functions. In practice, however, we are confined to a grid and thus cannot enforce quasiconvexity along every direction. Therefore we relax the notion of directional quasiconvexity to only a finite set of directions. Doing so results in the notion of \textit{approximate quasiconvexity}. We can quantify the degree to which a function might lack quasiconvexity, expressed in terms of the directional resolution which we define below.
\begin{definition}[Directional resolution]
Let $\Dir=\{d_1,\dots,d_N\}$ be a set of unit vectors. Then we define the \emph{directional resolution} of $\Dir$ as
\bq\label{eq:dir_res}
d\theta \equiv \max_{\abs{w}=1}\min_{d\in \Dir }\cos^{-1}(w^\intercal d)
\eq
\end{definition}
it the largest angle an arbitrary unit vector can make with any vector in $V$. In two dimensions $d\theta$ is simply half the maximum angle between any two direction vectors.

Recalling that a necessary condition for a function $u$ to be quasiconvex is $-\LQC[u]\le0$ in the viscosity sense \cite{barron2013quasiconvex}, we have the following result.

\begin{proposition}[Approximate quasiconvexity]
\label{ApproxQC}
Let $u$ be a $C^2$ function defined on $\Omega\subset\Rn$. Let $\Dir$ be a set of directions, with directional resolution $d\theta\leq\frac{\pi}{4}$. Also, suppose $u$ is quasiconvex along every $d\in \Dir$. Then we have the following approximate quasiconvexity estimate
\[
-\LQC[u]\leq\mathcal{O}(d\theta).
\]
\end{proposition}

\begin{proof}
Without loss of generality, assume $x=0$ and $u(0)=0$, We may also assume $u$ is locally quadratic,  $u(x)=x^\intercal A x+b^\intercal x$, for some real, symmetric matrix $A$, and for $b\in\Rn$. 
Next, suppose $w$ is an arbitrary unit vector satisfying $\nabla u(0)^\intercal w=b^\intercal w=0$. Decompose $w$ as follows:
\[
w = \cos(\theta)v + \sin(\theta) p
\]
where $v\equiv\text{argmin}_i(\cos^{-1}(w^\intercal v_i))$, $\theta = \cos^{-1}(w^\intercal v)$, and $p$ is some unit vector orthogonal to $v$. By hypothesis $\theta\le\pi/4$. Taking $\lambda_n$ to denote the largest eigenvalue of $A$, we observe:
\begin{align*}
0=u(0) &= u\left(\frac{1}{2}(-\cos(\theta) v)+\frac{1}{2}(\cos(\theta) v)\right)\\
&\leq \max(u(-\cos(\theta) v),u(\cos(\theta) v))\quad\text{(by quasiconvexity)}\\
&= \max(u(- w+\sin(\theta)p),u( w-\sin(\theta)p))\\
&=w^\intercal Aw+\sin^2(\theta)p^\intercal Ap-2\sin(\theta)w^\intercal Ap+\sin(\theta)|p^\intercal b|\\
&\leq w^\intercal Aw+ \sin^2(\theta)\lambda_n + 2\sin(\theta)\lambda_n + \sin(\theta)\abs{b}\\
&\leq w^\intercal Aw+ \sin^2(d\theta) \lambda_n + 2\sin(d\theta)\lambda_n+ \sin(d\theta)\abs{b}\\
&=w^\intercal Aw +  \mathcal{O}( d\theta)
\end{align*}
In the above calculations, we used the fact that $\phi(y)\geq u(y)$ and $y^\intercal Ay\leq\lambda_n$ for every $y$, as well as the hypothesis that $b^\intercal w=0$. Taking the minimum over all unit vectors $w$ satisfying $w^\intercal b=0$ yields the desired result.
\end{proof}

%%%%%%%%%%%%%%%%%%%%%%%%%%%%%%%%%%%%%%%%%%%%
%%%%%%%%%%%%%%%%%%%%%%%%%%%%%%%%%%%%%%%%%%%%
\section{Convergent Finite Difference Schemes}\label{sec:numerical_schemes}
%%%%%%%%%%%%%%%%%%%%%%%%%%%%%%%%%%%%%%%%%%%%
%%%%%%%%%%%%%%%%%%%%%%%%%%%%%%%%%%%%%%%%%%%%

In what follows, we provide numerical schemes which discretize the above PDEs. As we will see, the schemes we present here fall into a general class of degenerate elliptic finite difference schemes for which there exists a convergence framework.

Before we begin we introduce some notation, which we will carry throughout the rest of the article. We will assume we are working on the hypercube $[-1,1]^n\subset\Rn$. We write $x=(x_1,\dots,x_n)\in[-1,1]^n$. For simplicity, we discretize the  domain $[-1,1]^n$ with a \emph{uniform} grid, resulting in the following \emph{spatial resolution}:
\[
h\equiv\frac{2}{N-1},
\]
where $N$ is the number of grid points used to discretize $[-1,1]$. Note that we use $h$ here to denote the spatial resolution, and is not to be confused with the $h$ in the formulations of \eqref{eQC}, \eqref{eQCE}.

We use the following notation for our computational domain:
\[
\Omega^h \equiv [-1,1]^n\cap h\mathbb{Z}^n
\]
We define a \emph{grid vector}, $v$, as follows:
\[
\label{gridvector}
%v \equiv \frac{x-y}{\abs{x-y}},\quad x,y\in\Omega^h
v \equiv \frac{x-y}{h},\quad x,y\in\Omega^h
\]

%%%%%%%%%%%%%%%%%%%%%%%%%%%%%%%%%%
\subsection{Finite difference equations and wide stencils}
%%%%%%%%%%%%%%%%%%%%%%%%%%%%%%%%%%

Consider a degenerate elliptic PDE, $F[u]=0$, and its corresponding finite difference
equation
\[
F_\rho[u]=0,\quad x\in\Omega^h
\]
where $\rho$ is the discretization parameter. In our case, we take $\rho=(h,d\theta)$ for the schemes presented hereafter. 

In general, $F_\rho[u]:C(\Omega^h)\to C(\Omega^h)$, where $C(\Omega^h)$ is the set of grid functions $u:\Omega^h\to\R$. We assume the following form:
\[
F_\rho[u](x) = F_\rho(u(x),u(x)-u(\cdot))\quad\ \text{for}\ x\in\Omega^h
\]
where $u(\cdot)$ corresponds to the value of $u$ at points in $\Omega^h$. For a given set of grid vectors, $V$, we say $F_\rho[u]$ has \emph{stencil width} $W$ if for any $x\in\Omega^h$, $F_\rho[u](x)$ depends only on the values $u(x + hv)$ where $v\in V$ and $\max_{v\in V}\norm{v}_\infty\le W$. We assume stencils are symmetric about the reference point, $x$, and have at least width one.

\begin{example}
	In two dimensions, the directional resolution can be written as $d\theta=\arctan(1/W)/2$. A stencil of width $W=1$ would correspond to the vectors $V=\{(0,\pm1), (\pm1,0), (\pm1,\pm1), (\mp1,\pm1)\}$.
\end{example}

The Dirichlet boundary conditions given by \eqref{dirichlet} are incorporated by setting
\[
F_\rho[u](x) = u(x)-g(x)\quad\ \text{for}\ x\in\partial\Omega^h
\]

The PDEs we discuss involve a second-order directional derivative and a directional Eikonal operator. We use the following standard discretizations for our schemes:
\begin{align*}
\abs{u_v^h(x)} &\equiv \max\left\{\frac{u(x \pm hv)-u(x)}{h}\right\}\\
u_{vv}^h(x) &\equiv \frac{u(x+hv)+u(x-hv)-2u(x)}{h^2} 
\end{align*}
where $u_v^h$ and $u_{vv}^h$ denote the finite difference approximations for the first and second order directional derivatives in the grid direction $v$. Recall that for any twice continuously differentiable function $u$, the standard Taylor series computation yields the following consistency estimate: 
\begin{align}
\abs{u_v^h(x)} &=  \abs{\nabla u(x)^\intercal v} + \bO(h) \label{uv}\\
u_{vv}^h(x) &= v^\intercal D^2u(x) v + \bO(h^2) \label{uvv} 
\end{align}

%%%%%%%%%%%%%%%%%%%%%%%%%%%%%%%%%%
\subsection{A finite difference method for the PDE}
%%%%%%%%%%%%%%%%%%%%%%%%%%%%%%%%%%

We begin by considering the full discretization of $F^\e[u]$, denoted $F^\e_{h,d\theta}[u]$, where $h$ is the spatial resolution and $d\theta$ is the directional resolution of the set of grid vectors $V$. We use the spatial discretizations given by \eqref{uv}, \eqref{uvv}:
\bq\label{eq:Fscheme}
F^\e_{h,d\theta}[u](x)\equiv \min_{v\in V}\left\{\frac{1}{\e}\abs{u_v^h(x)}+u_{vv}^h(x)\right\}\quad \text{for $x\in\Omega^h$}
\eq

We write the full discretization of \eqref{approxOb} as follows:
\bq 
\label{approxObscheme}
\begin{cases}
\max\{u(x)-g(x),\e - F^\e_{h,d\theta}[u](x)\} = 0 & x\in\Omega^h\\
u(x) = g(x) & x\in\partial\Omega^h
\end{cases}
\eq
where we make the following abbreviation:
\[
G^\e_{h,d\theta}[u]\equiv\max\left\{u-g,\e-F^\e_{h,d\theta}[u]\right\}
\]

\subsection{Iterative solution method}

We implement a fixed-point solver to recover the numerical solution of \eqref{approxObscheme}.  The method is equivalent to a Forward Euler time discretization of the parabolic version of \eqref{eQCE}, 
\[
u_{t} + G^\e[u]=0\quad \text{in $\Omega$}
\]
along with $u = g$ on $\partial\Omega$ and $u = u_0$ for $t=0$.

In particular the following iterations are performed until a steady state is reached:
\[
\begin{cases}
u^{n+1} = u^{n} - \delta G^\e_{h,d\theta}[u^{n}], & n\ge0\\
u^{0} = u_0
\end{cases}
\]
where satisfies the CFL condition $\delta\le 1/K^{h,\e}$ (see \cite{ObermanSINUM}) and $K^{h,\e}$ is the Lipschitz constant of the scheme. For our scheme, $K$ is given by 
\[
K^{h,\e} = \frac{1}{\e h} + \frac{1}{h^2}
\]
where the first and second term come from the discretizations of the first and second order directional derivatives, respectively.
%%%%%%%%%%%%%%%%%%%%%%%%%%%%%%%%%%
\subsection{Using a first-order discretization}\label{sec:first_order}
%%%%%%%%%%%%%%%%%%%%%%%%%%%%%%%%%%
In higher dimensions we need the second order term in \eqref{eQC} to guarantee quasiconvexity of the solutions. Interestingly, however, we will see numerically all that is required is the discretization of the first order term. This is illustrated in the following computation:

\begin{align*}
-\min_{v\in V}\frac{\abs{u_v^h}}{\e}&=-\min_{v\in V}\left\{\frac {\max\{u(x\pm hv)-u(x)\}} {\e h}\right\}\\
&= -\min_{v\in V}\left\{ \frac{\abs{\grad u(x)^\intercal v}}{\e} + \frac{h}{2\e}v^\intercal D^2u(x)v \right\}+ \bO\left(\frac{h^2}{\e}\right)
\end{align*}
Choosing $\e = h/2$ results in the following:
\[
-\min_{v\in V}\frac{\abs{u_v^h}}{\e} = -\min_{v\in V}\left\{ \frac{\abs{\grad u(x)^\intercal v}}{h/2} + v^\intercal D^2u(x)v \right\}+ \bO(h)
\]
Thus, for $\e=h/2$, we see for fixed $(h,d\theta)$ that the discretization of the first order term approximates $F^\e[u]$.

\section{Convergence of numerical solutions} 
In this section we introduce the notion of degenerate elliptic schemes and show that the solutions of the proposed numerical schemes converge to the solutions of \eqref{approxOb} as the discretization parameters tend to zero. The standard framework used to establish convergence is that of Barles and Souganidis \cite{BS91}, which we state below. In particular, it guarantees that the solutions of any monotone, consistent, and stable scheme converge to the unique viscosity solution of the PDE.

\subsection{Degenerate elliptic schemes}
Consider the Dirichlet problem for the degenerate elliptic PDE, $F[u]=0$, and recall its corresponding finite difference formulation:
\[
\begin{cases}
F_\rho(u(x),u(x)-u(\cdot)) = 0 & x\in\Omega^h\\
u(x)-g(x) = 0 & x\in\partial\Omega^h
\end{cases}
\]
where $\rho$ is the discretization parameter.
\begin{definition}
$F_\rho[u]$ is a \emph{degenerate elliptic scheme} if it is non-decreasing in each of its arguments.
\end{definition}

\begin{remark}
Although the convergence theory is originally stated in terms of \emph{monotone} approximation schemes (schemes with non-negative coefficients), ellipticity is an equivalent formulation for finite difference operators \cite{ObermanSINUM}.
\end{remark}

\begin{definition}
The finite difference operator $F_\rho[u]$ is \emph{consistent} with $F[u]$ if for any smooth function $\phi$ and $x\in\Omega$ we have
\[
F_\rho[\phi](x)\to F[\phi](x)\quad \text{as $\rho\to0$}
\]
\end{definition}

\begin{definition}
The finite difference operator $F_\rho[u]$ is \emph{stable} if there exists $M>0$ independent of $\rho$ such that if $F_\rho[u]=0$ then $\| u \|_\infty\le M$.
\end{definition}

\begin{remark}[Interpolating to the entire domain]
The convergence theory assumes that the approximation scheme and the grid function are defined on all of $\Omega$. Although the finite difference operator acts only on functions defined on $\Omega^h$, we can extend such functions to $\Omega\setminus\Omega^h$ via piecewise interpolation. In particular, performing piecewise \emph{linear} interpolation maintains the ellipticity of the scheme, as well as all other relevant properties. Therefore, we can safely interchange $\Omega^h$ and $\Omega$ in the discussion of convergence without any loss of generality.
\end{remark}
\subsection{Convergence of numerical approximations}
Next we will state the theorem for convergence of approximation schemes, tailored to elliptic finite difference schemes, and demonstrate that the proposed schemes fit in the desired framework. In particular, we will show that the schemes are elliptic, consistent, and have stable solutions.
\begin{proposition}[Convergence of approximation schemes \cite{BS91}]
\label{BSconvergence}
Consider the degenerate elliptic PDE, $F[u]=0$, with Dirichlet boundary conditions for which there exists a strong comparison principle. Let $F_\rho[u]$ be a consistent and elliptic scheme. Furthermore, assume that the solutions of $F_\rho[u]=0$ are bounded independently of $\rho$. Then $u^\rho\to u$ locally uniformly on $\Omega$ as $\rho\to0$. 
\end{proposition}

\begin{lemma}[Ellipticity] 
\label{ellipticity}
The scheme given by \eqref{approxObscheme} is elliptic.
\end{lemma}
\begin{proof}
The negative of each of the spatial discretizations above are elliptic for fixed $v$, so taking the minimum over all $v$ maintains the ellipticity. Therefore, the discretization of $F^\e_{h,d\theta}[u]$ is elliptic. Moreover, the obstacle term is trivially elliptic. Hence $G^\e_{h,d\theta}[u]$, being the maximum of two elliptic schemes, is also elliptic.
\end{proof}

\begin{lemma}[Consistency]
\label{consistency}
Given a smooth function $u$ and a set of grid vectors $V$ with directional resolution $d\theta\le\pi/4$, we have the following estimate:
\bq\label{consistency_estimate}
G^\e_{h,d\theta}[u] - G^\e[u] = \bO(h+d\theta)
\eq
In other words, the scheme \eqref{approxObscheme} is consistent with \eqref{approxOb}.
\end{lemma}

\begin{proof}
It suffices to show that $F^\e_{h,d\theta}[u] - F^\e[u] = \bO(h+d\theta)$. Fix $x\in\overline\Omega$ and let $u$ be a smooth function on $\overline{\Omega}$. Define
\[
w\equiv\argmin_{\abs{v}=1}\left\{ \frac{1}{\e}\abs{\nabla u(x)^\intercal v} + v^\intercal D^2u(x) v\right\}
\]
Let $v\in V$ and make the following decomposition:
\[
v = \cos(\theta)w + \sin(\theta)p
\]
for some unit vector $p$ orthogonal to $w$. Performing a similar calculation to the proof in Proposition \ref{ApproxQC}, we obtain the following consistency estimate from the directional resolution:
\[
F^\e[u] \le\min_{v\in V}\left\{\frac{1}{\e}\abs{\nabla u^\intercal v} + v^\intercal D^2u v\right\} \le F^\e[u] +\bO(d\theta)
\]
Finally, we conclude \eqref{consistency_estimate} by recalling the standard Taylor series consistency estimate from equations \eqref{uv}, \eqref{uvv}.
\end{proof}

%\begin{remark}[Ensuring quasiconvexity]
%\label{rmk:correction}
%\blue{remove this, save it, and add to the list of stuff to prove with Carlier-Galichon}
%
% Taking a closer look at the proof above, we see that we can rewrite \eqref{consistency_dr} and apply it to the PDE in \eqref{approxOb} as follows
% \[
%\sqrt\e -F^\e[u] \le \frac{C}{\sqrt\e}\sin(d\theta) + (\sin^2(d\theta) + \sin(2d\theta))\lambda_n[u]+\sqrt\e-F^\e_{d\theta}
%\]
%where $F^\e_{d\theta}[u]$ denotes the discretization of $F^\e[u]$ with no spatial resolution. In particular, for a fixed $\e$, a subsolution of 
%\[
%\frac{C}{\sqrt\e}\sin(d\theta) + (\sin^2(d\theta) + \sin(2d\theta))\lambda_n[u]+\sqrt\e-F^\e_{d\theta}=0
%\]
%is also a subsolution of $\sqrt\e-F^\e[u]=0$. Hence, with a slight perturbation of $\sqrt\e-F^\e_{d\theta}[u]=0$, one can attain a quasiconvex function by enforcing quasiconvexity in finitely many directions. Computationally, this gives us the choice of forfeiting a wider stencil for a larger regularization term. 
%
%\end{remark}

Next, we establish stability of our numerical solutions by applying a discrete comparison principle for strict sub- and supersolutions. In particular, we use the following result.

\begin{proposition}[Discrete comparison principle for strict subsolutions] \label{discreteCP}
Let $F_\rho[u]$ be a degenerate elliptic scheme defined on $\Omega^h$. Then for any grid functions $u,v$ we have: 
\[
F_\rho[u]<F_\rho[v]\implies u\le v\quad\text{in $\Omega^h$}
\]
\end{proposition}
\begin{proof}
Suppose for contradiction that
\[
\max_{x\in\Omega^h}\{u(x) - v(x)\} = u(x_0) - v(x_0) > 0
\]
Then
\[
u(x_0) - u(x) > v(x_0) - v(x)\quad\text{for every $x\in\Omega^h$}
\]
and so by ellipticity (Lemma \ref{ellipticity})
\[
F_\rho[u](x_0)\ge F_\rho[v](x_0)
\]
which is the desired contradiction.
\end{proof}

\begin{lemma}[Stability]
\label{stability}
The numerical solutions of \eqref{approxObscheme} are stable.
\end{lemma}

\begin{proof}
%This a direct application of the discrete comparison principle (see \cite{ObermanSINUM}). Indeed, suppose $u$ is a numerical solution, and define $g_1\equiv \min_{x\in\overline\Omega}g(x)$ and $g_2\equiv\max_{x\in\overline\Omega}g(x)$. Then $g_1$ and $g_2$ are viscosity sub- and supersolutions, respectively. Therefore by the discrete comparison principle we have the following bound
%\[
%g_1\le u(x)\le g_2 \quad\text{for}\ x\in\Omega^h
%\]
%Moreover, this bound is independent of $h$ and $d\theta$.

This a direct application of Proposition \ref{discreteCP}. Indeed, suppose $u$ is a numerical solution such that $G^\e_{h,d\theta}[u]=0$. Define $g_1(x)\equiv C+2\e^2 \sum_{i=1}^nx_i$, where $C\in\R$ is chosen so that $g_1(x)<g(x)$ for every $x\in\overline\Omega$. Moreover one can check that $(g_1)_{vv}^h(x)=0$. Then we observe:
\begin{align*}
G^\e_{h,d\theta}[g_1](x) &= \max\left\{g_1(x)-g(x),\e-\min_{v\in V}\left\{ \frac{\max\{g_1(x\pm hv)\}-g_1(x)}{\e h} \right\}\right\}\\
&= \max\left\{g_1(x)-g(x),\e-2\e\min_{v\in V}\left\vert\sum_{i=1}^nv_i\right\vert\right\}\\
&= \max\left\{g_1(x)-g(x),-\e\right\}
\end{align*}
where the last equality follows from the general assumption that the stencil width is at least one with directional resolution $d\theta\le\pi/4$.
Next, define $g_2(x)\equiv\max_{x\in\overline\Omega}g(x)$. Then in summary we have that $G^\e_{h,d\theta}[g_1]<0$ and $G^\e_{h,d\theta}[g_2]\ge0$. Therefore by Proposition \ref{discreteCP},
\[
\min_{x\in\overline\Omega} g_1(x)\le u(x)\le g_2 \quad\text{for}\ x\in\Omega^h
\]
Moreover, this bound is independent of $h$ and $d\theta$.

\end{proof}

\begin{proposition}\label{prop:convergence}
The solutions of \eqref{approxObscheme} converge locally uniformly on $\Omega$ to the solutions of \eqref{approxOb} as $h,d\theta\to0$.
\end{proposition}

\begin{proof}
By Lemmas (\ref{ellipticity}, \ref{consistency}, \ref{stability}) we see that \eqref{approxObscheme} is a consistent and elliptic scheme which has stable solutions. Moreover, a weak comparison principle holds by Proposition \ref{e-CP}. Therefore by Proposition \ref{BSconvergence}, the numerical solutions converges uniformly on compact subsets of $\Omega$.
\end{proof}

%%%%%%%%%%%%%%%%%%%%%%%%%%%%%%%%%%%%%%%%%
%%%%%%%%%%%%%%%%%%%%%%%%%%%%%%%%%%%%%%%%%
\section{Numerical results}
%%%%%%%%%%%%%%%%%%%%%%%%%%%%%%%%%%%%%%%%%
%%%%%%%%%%%%%%%%%%%%%%%%%%%%%%%%%%%%%%%%%
In this section we present numerical results. 
The tolerance for the fixed-point iterations was taken to be $10^{-6}$ and unless otherwise stated we set $\e=h/2$, where $h$ is the spatial resolution of the grid.

Technical conditions on $g$ given by Remark \ref{rmk:bc} are needed to ensure that the boundary conditions are held in the strong sense. In practice we violated this condition, resulting in solutions which were discontinuous at the boundary. 

Two natural choices for initialization of the solution are: (i) the obstacle function, $g(x)$, (ii) a quasiconvex function below $g$.  We found that the first choice leads to very slow convergence.  In particular, the parabolic equation takes time  $\bO(1/\e)$ to converge.  See  Example~\ref{ex_initabove} below.  On the other hand, the second choice results in faster convergence: the simplest choice is simply the constant function with value the minimum of $g$. After one step of the iteration the boundary values are attained. Moreover, starting with this choice allows us to use the iterative method with $\e = 0$  to find the quasiconvex envelope.

\begin{remark}
It is an open question whether solve the $\e=0$ time-dependent PDE with quasiconvex initial data below $g$ will converge.  To apply our convergence results, we take $\e = h/2$. 
\end{remark}

\subsection{Results in one dimension}
We present examples demonstrating the convergence of approximate quasiconvex envelopes to the true quasiconvex envelope as $\e\to0^+$.  We also compare the iteration count when starting from below and at the obstacle.

\begin{example}[Convergence as $\e\to0$]
\label{ex_epsconvergence}
We consider the convergence of numerical solutions of \eqref{approxObscheme} as $\e\to 0$ with the following obstacle function:
\[
g(x)=\min\{\abs{x-0.5},\abs{x+0.5}-0.3\}
\]
The results are displayed in \autoref{prelim_example}. Indeed, as expected we see convergence to the true quasiconvex envelope, from below, as $\e\to0$. 
\end{example}

\begin{example}[Visualization of iterations]
\label{ex_initabove}
Next, we consider the following obstacle:
\[
g(x)= -x^2+1
\]
whose quasiconvex envelope is simply $g_2^{QCE}\equiv0$. We demonstrate the evolution of the iterations when the initial data is taken to be the obstacle.  In this case, the solution corresponded closely to $u(x,t) = \min(g(x), c - \e t)$ where $c = \max g$.   This illustrates the slow speed of convergence, and the fact that the equation degenerates to a trivial operator as  $\e \to 0$. 
Results are displayed in \autoref{prelim_example}.
\end{example}
\begin{figure}[t]
\centering
\includegraphics[width=.49\textwidth]{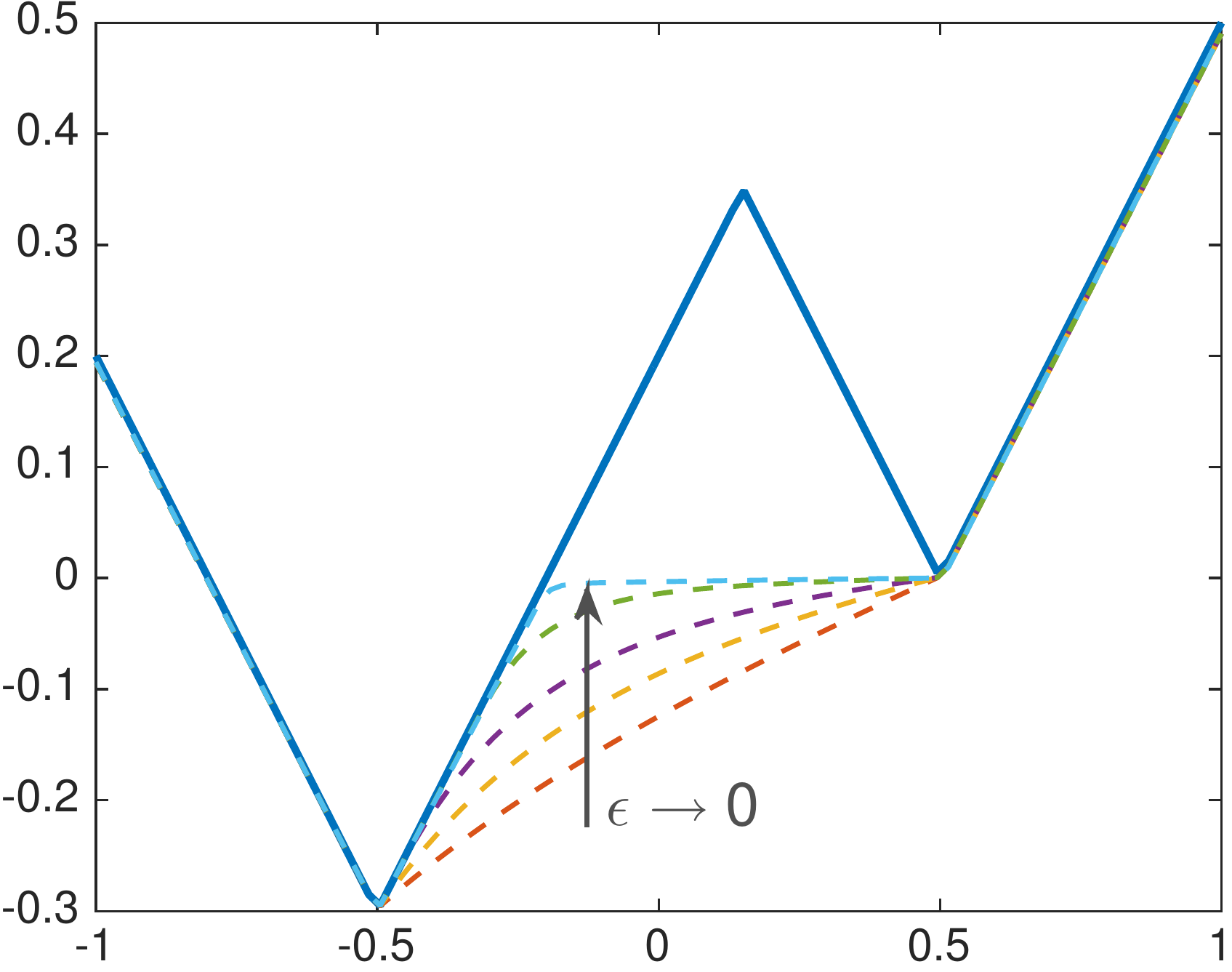}
\includegraphics[width=.49\textwidth]{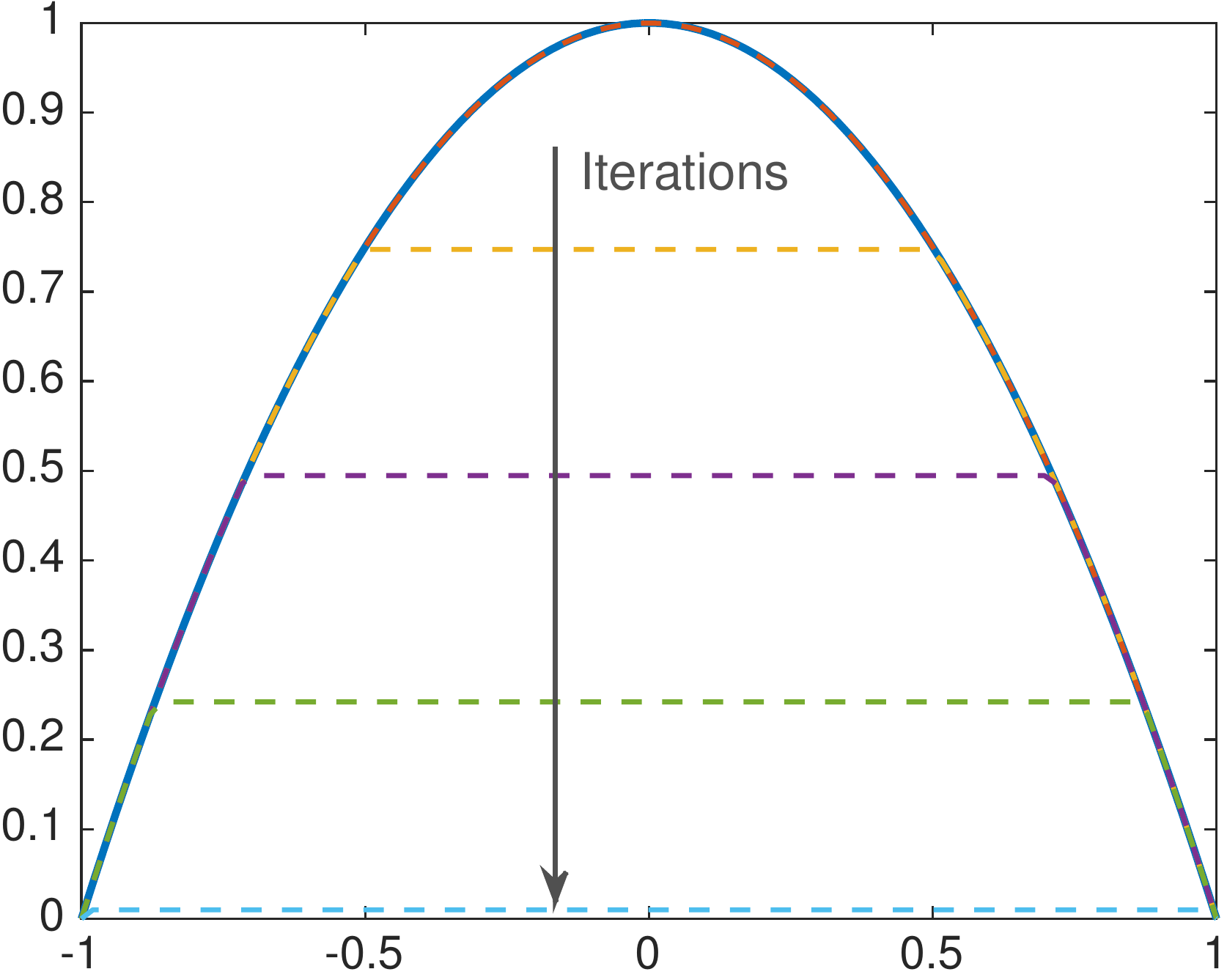}
\caption{Obstacles (solid) and the numerical solutions (dashed). Left: Example \ref{ex_epsconvergence}. Convergence in $\e$. Bottom to top: $\e = 0.2\,,0.1\,,0.05\,,0.001\,,0.0001$. Right: Example \ref{ex_initabove}. Fixed-point iterations when $u^0 \equiv g_2$.}
\label{prelim_example}
\end{figure} 

%%%%%%%%%%%%%%%%%%%%%%%%%%%%%%%%%%%%%%%%%%%%%%%%%%%%%

\subsection{Numerical results in two dimensions}
All examples take the initial function in the iterative scheme to be the minimum value $g_{min} = \min g$ of the obstacle function. 
In all contour plots, the solid line represents the level sets of the original function and the dashed line represents the same level sets of the numerical solution. Unless otherwise stated, the two-dimensional numerical solutions shown are computed on a $64\times64$ grid, for illustration purposes.  Computations were performed on larger sized grids. 

We performed the same numerical experiments in \cite{AbbasiLineSolverQuasiConvex}, and achieved similar results. This is expected, since solving our PDE with small $\e$ is consistent with the QC envelope.  We do not reproduce those results here, in order to save space.

The examples we present focus on the difference between the two operators for values of $e$ close to 1. 

\begin{example}[Strict convexification of level sets]
\label{ex:squarelevelsets}
Let $g(x)$ be the signed distance function (negative on the inside) to the square, $S = \left\{x \mid \max_{i}\abs{x_i}=1/2 \right\}$.
% and let 
%\begin{align*}
%g(x) &\equiv 
%\begin{cases}
%-\text{dist}(x,\Gamma) & x\in\text{int}(\Gamma)\\
%\text{dist}(x,\Gamma) & x\notin\text{int}(\Gamma)
%\end{cases}
%\end{align*}
%where $\text{int}(\Gamma)$ denotes the interior of $\Gamma$ and $\text{dist}(x,\Gamma)\equiv \min_{y\in\Gamma}\abs{x-y}$ denotes the distance function. 

Although $g$ is convex, its level sets are not \emph{strictly} convex. We apply the iterative procedure to $g$ with $\e=0.5$ so that we strictly convexify the level sets. The results demonstrating this are displayed in \autoref{fig:gridaligned}.

\end{example}

\begin{example}[Non-convex signed distance function]
\label{ex:pacman}
In this example, $g(x)$ is the signed distance function  to the curve $\Gamma$, given below:
\begin{align*}
\Gamma(t) &= 
\frac{1}{2}\begin{cases} 
(\cos(t-\frac{\pi}{2}),\sin(t-\frac{\pi}{2})) & t\in[0,3\pi/2]\\
\left(0,\frac{t-7\pi/4}{\pi/4}\right) & t\in[3\pi/2,7\pi/4] \\
\left(\frac{7\pi/4-t}{\pi/4},0\right) & t\in[7\pi/4,2\pi]
\end{cases}
\end{align*}
We compare the results of using different $\e$. In particular, we choose $\e=h/2$ and $\e=1$. In the latter case, we see that the level sets are more curved. Results are displayed in \autoref{fig:gridaligned}.
\end{example}

\begin{example}
\label{ex:circles}
We consider the obstacle 
where $g$ is a cone with circular portions removed from its level sets.
Take 
\[
g(x) = 
\begin{cases}
1 &x\in A\\
\abs{x}-1/2 &x\in \Omega\setminus A
\end{cases}
\]
where $A=\{(x_1\pm1/2)^2+x_2^2\le 1/16\}\bigcup \{x_1^2+(x_2\pm1/2)^2\le1/16\}$. Results showing the $g=0.7$ level set are found in \autoref{fig:gridaligned}.
\end{example}

%Results are found in \autoref{fig:gridaligned}.
%
%
%\begin{example}[Grid-aligned cones]
%\label{ex:gridalignedcones}
%We also consider examples of the type:
%\[
%g^{\theta}(x)=\max\left\{\sqrt{(x_1+a^\theta_1)^2 + (x_2+a^\theta_2)^2},\sqrt{(x_1+b^\theta_1)^2 + (x_2+b^\theta_2)^2}\right\}
%\]
%where $R(\theta)$ is the counter-clockwise rotation matrix by angle $\theta$, $a^\theta = R(\theta)(0.5,0)^\intercal $, and $b^\theta = R(\theta)(-0.5,0)^\intercal $. The vertical translation of the second argument by $\alpha$ adds an additional non-convexity by imposing asymmetry about the plane $x=0$. For now, we consider the case $\theta=0$, so that the non-convexities lie along horizontal lines. Results are found in \autoref{fig:gridaligned}.
%
%\end{example}

%\begin{figure}[t]
%\centering
%\includegraphics[width=.4\textwidth]{example_square}
%\hspace{0.8cm}
%\includegraphics[width=.4\textwidth]{example_pacman}
%\caption{Examples \ref{ex:squarelevelsets} (left), \ref{ex:pacman} (right).}
%\label{fig:gridaligned}
%\label{example_square_pacman}
%\end{figure} 

\begin{figure}[t]
\centering
\includegraphics[width=.3\textwidth]{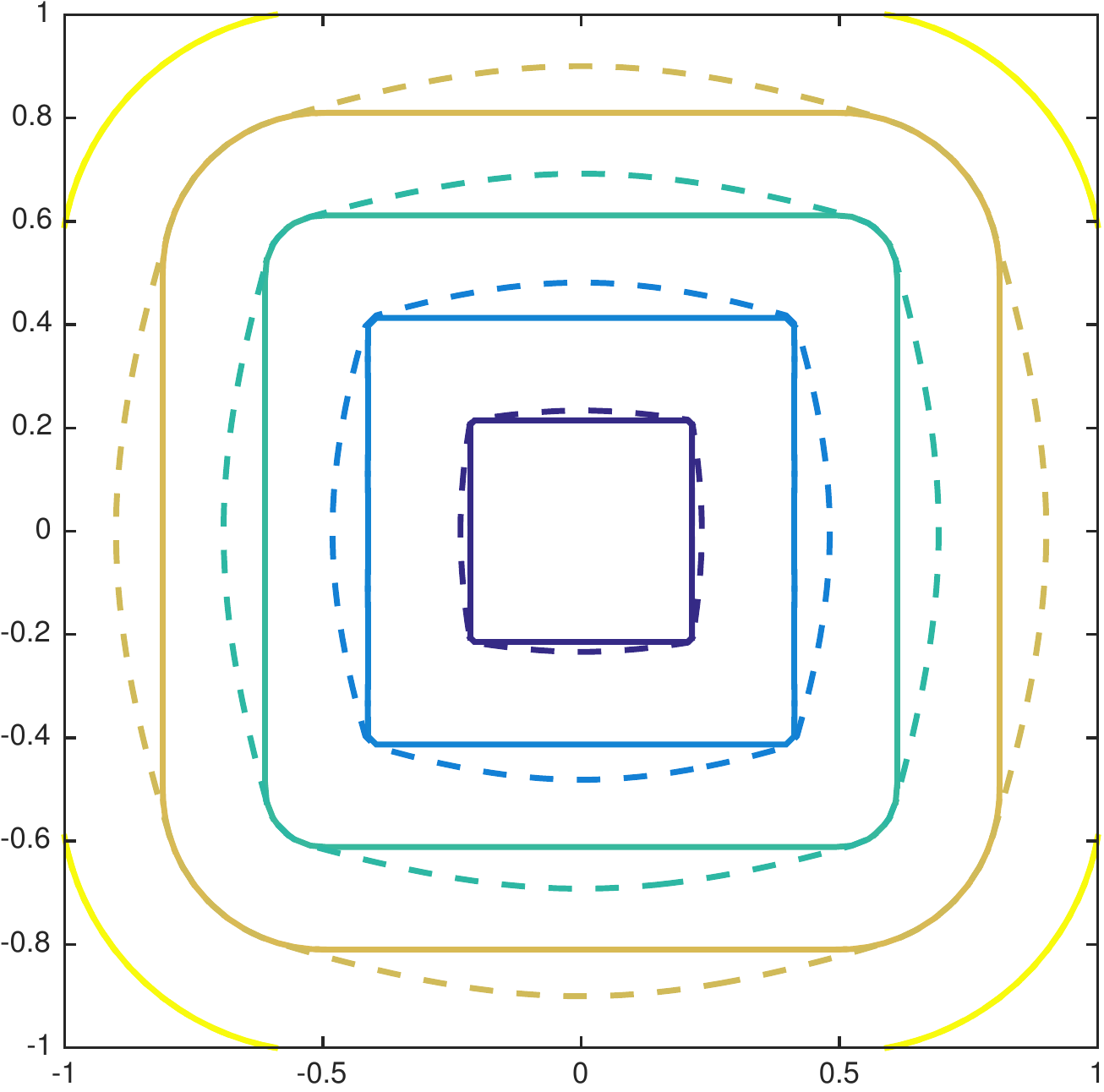}
%\hfill
%\hspace{.5cm}
\includegraphics[width=.3\textwidth]{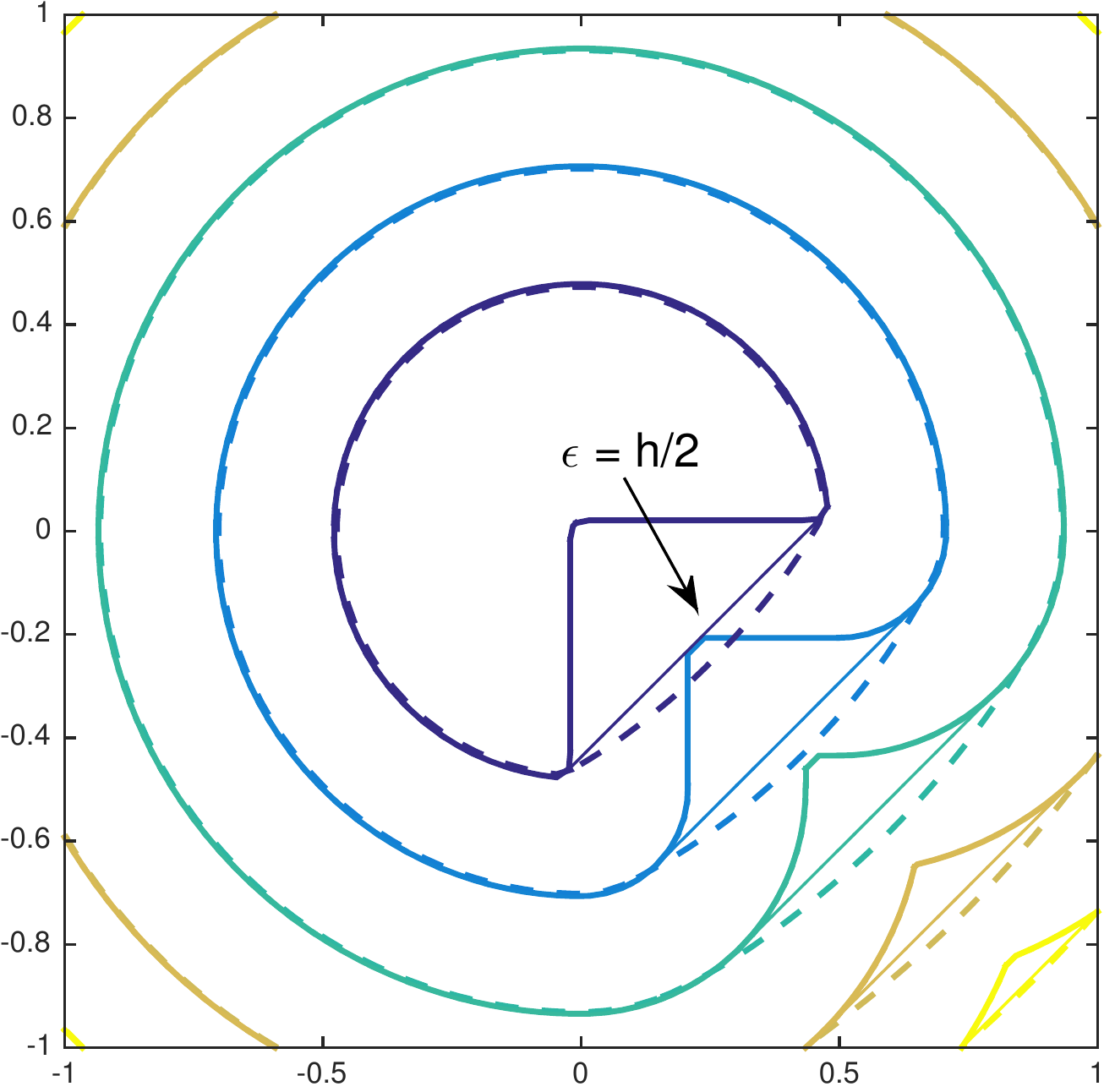}
%\hfill
%\hspace{.5cm}
\includegraphics[width=.345\textwidth]{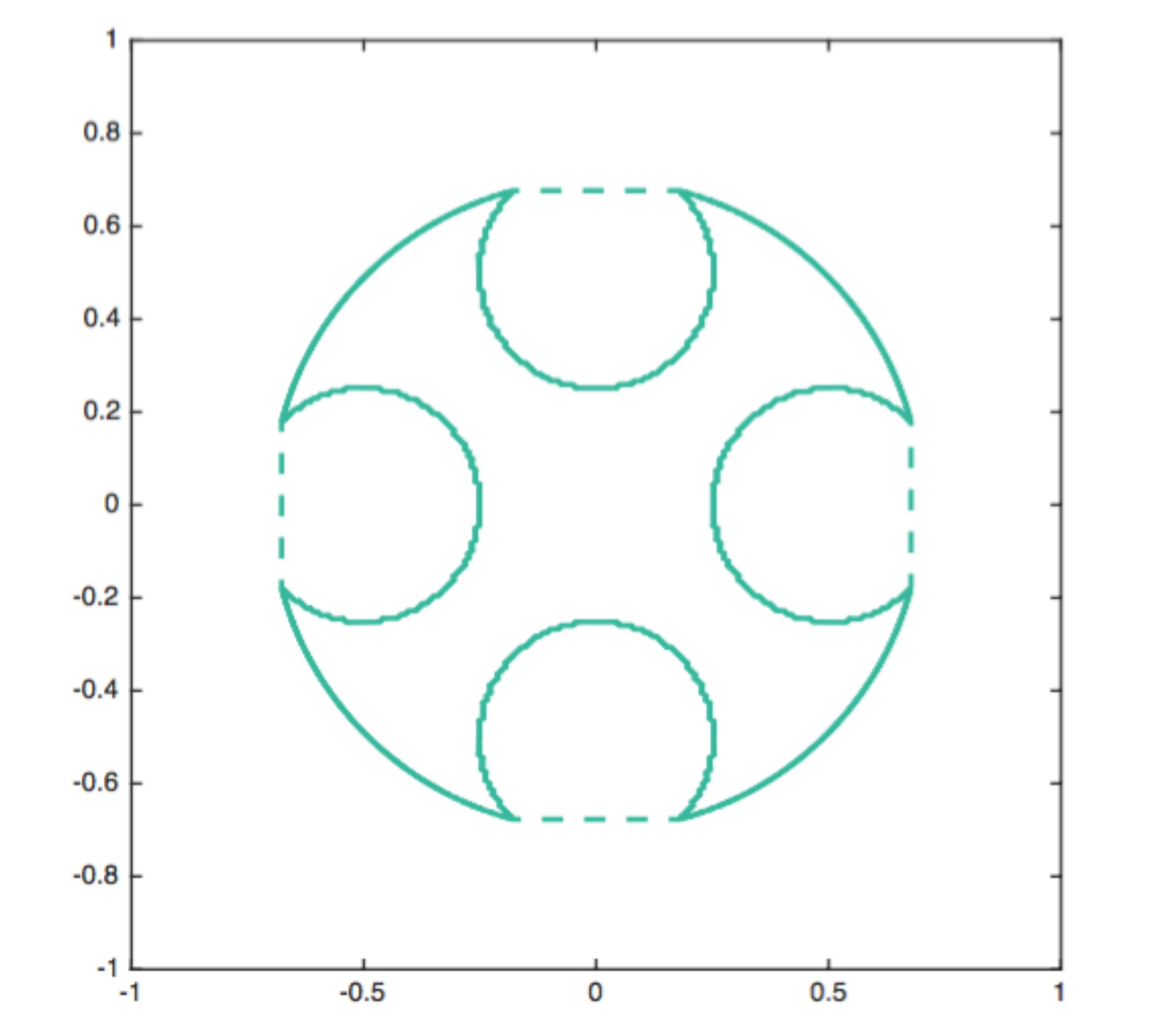}
\caption{Left: Examples \ref{ex:squarelevelsets}, square level sets become uniformly convex with $\e = 1/2$.  Middle: Ex~\ref{ex:pacman}, the solution with $\e=1$ is indicated by the dashed contours.  Right: Ex~\ref{ex:circles} with $\e$ near zero.  }
\label{fig:gridaligned}
\end{figure} 

%%%%%%%%%%%%%%%%%%%%%%%%%%%%%%%%%%%%%%%%%%%%%%%%%%%%%%%%%
\begin{example}[Comparison to the $\e$-robust quasiconvex envelope]\label{example_robust}
We compare the solutions of our PDE to the solutions of the line solver presented in \cite{AbbasiLineSolverQuasiConvex} which returns the robust quasiconvex envelope. In particular, we use a stencil width $W=5$, we used the same directions in the line solver, and we set $\e=0.02$ for the strict QCE, and we also used a regularization of  $.02$ for the robust QC (in principle we should have used $.02^2$ but the matching value was better for illustration purposes.

\begin{figure}[t]
\centering
\includegraphics[width=.3\textwidth]{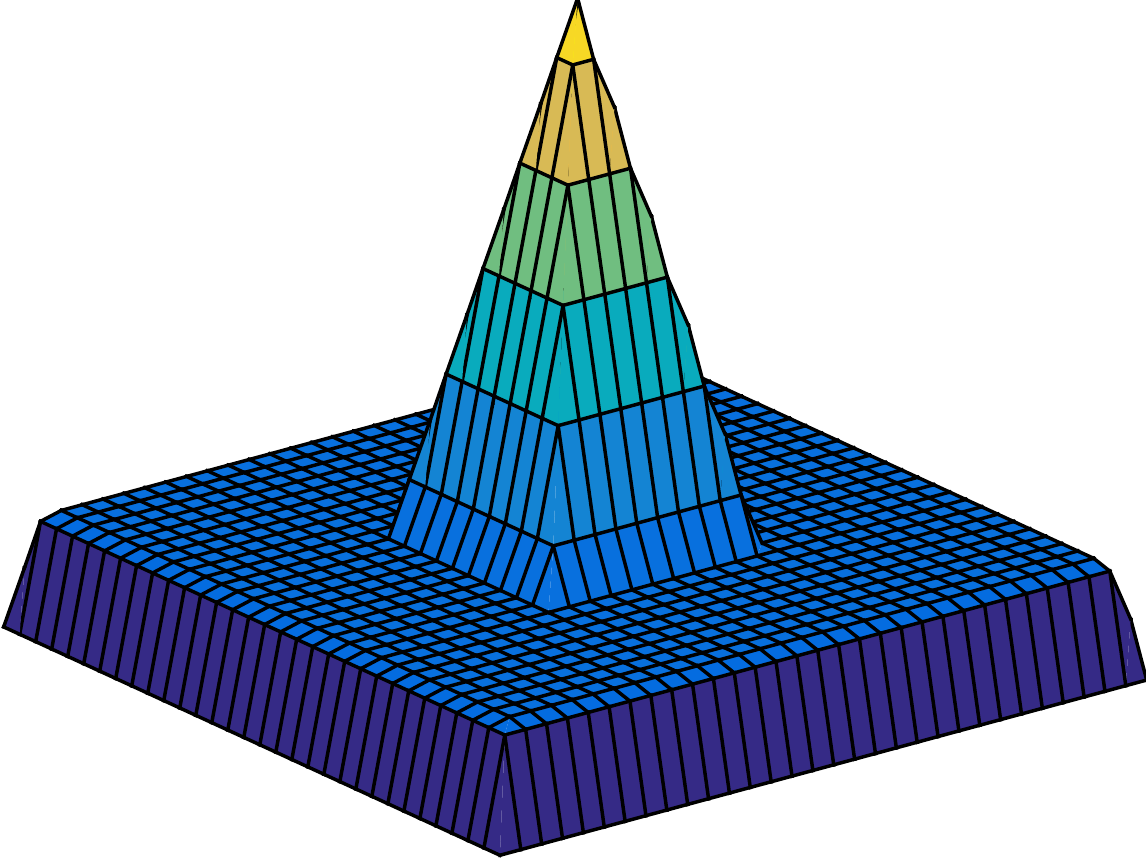}
%\hfill 
%\hspace{.5cm}
\includegraphics[width=.3\textwidth]{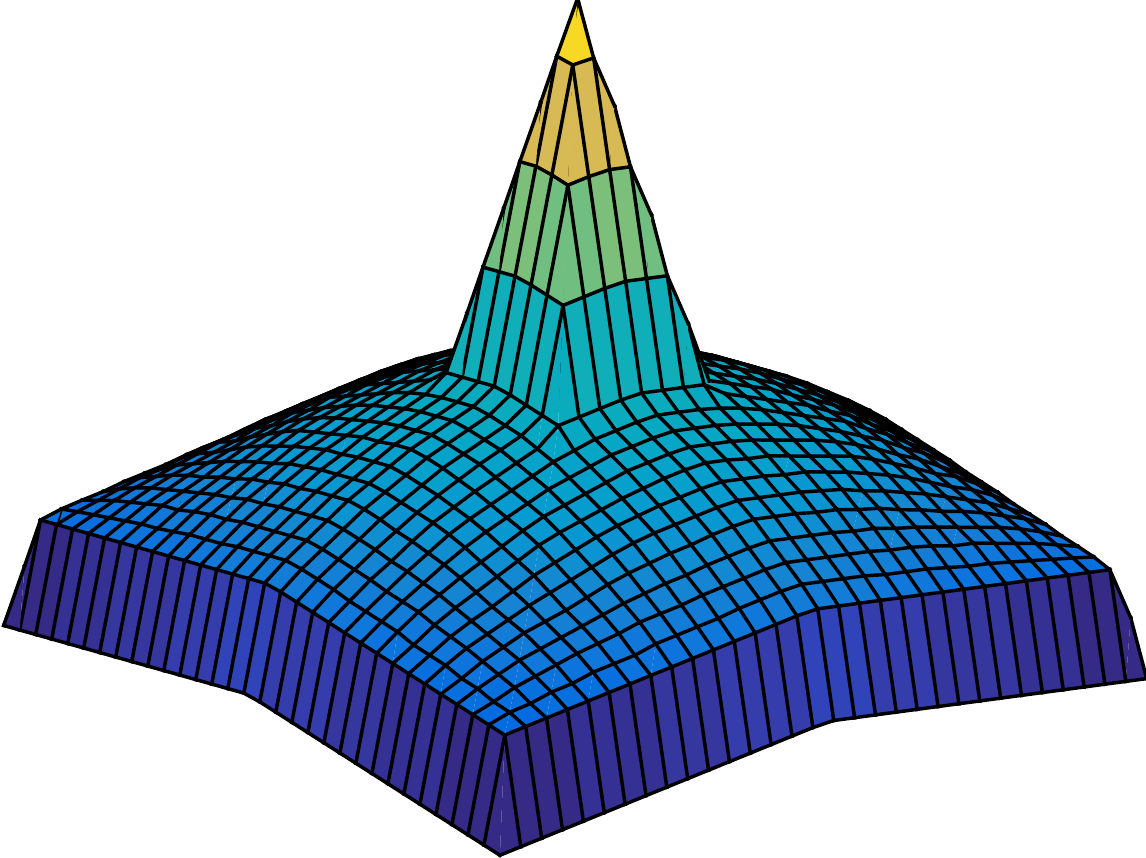}
%\hfill \hspace{.5cm}
\includegraphics[width=.3\textwidth]{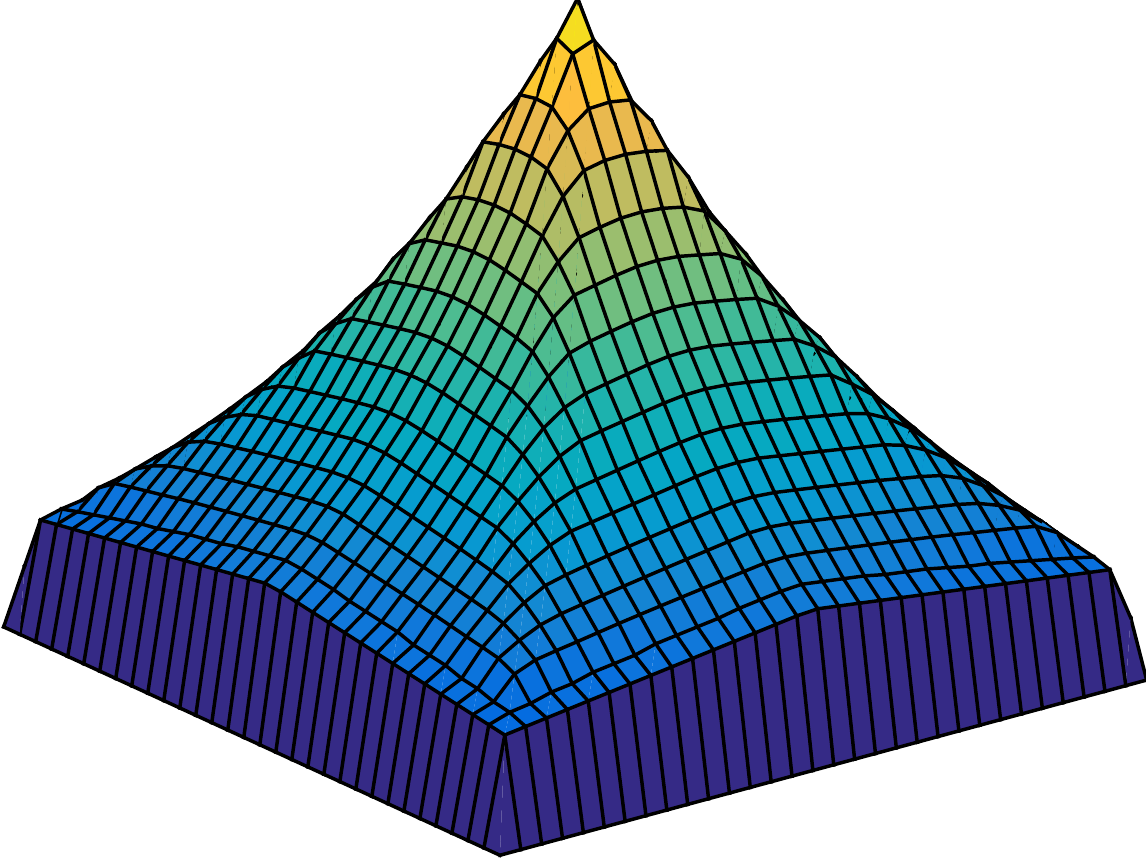}
\\
%\vspace{.6cm}
\includegraphics[width=.3\textwidth]{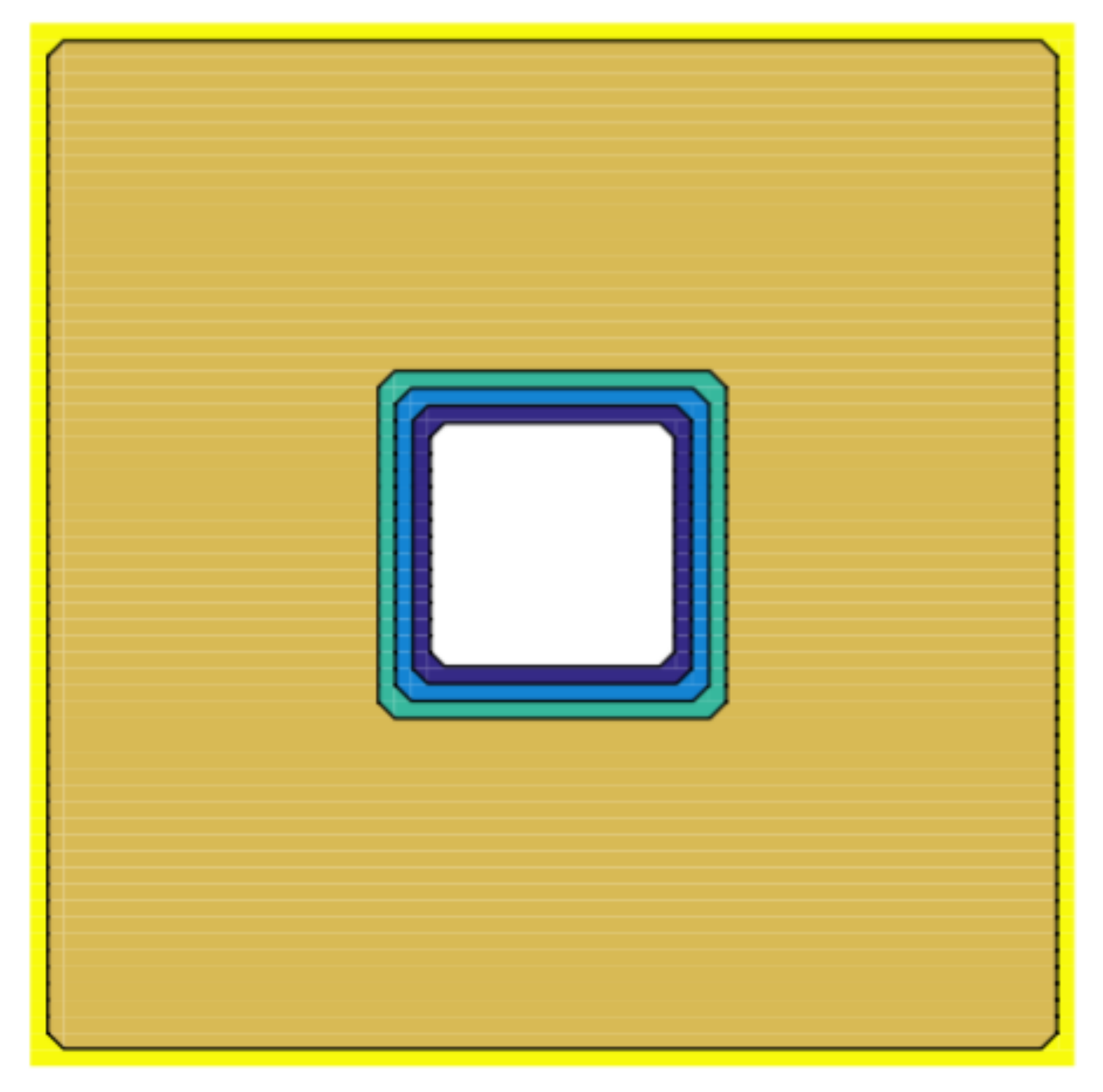}
%\hfill \hspace{2cm}
\includegraphics[width=.3\textwidth]{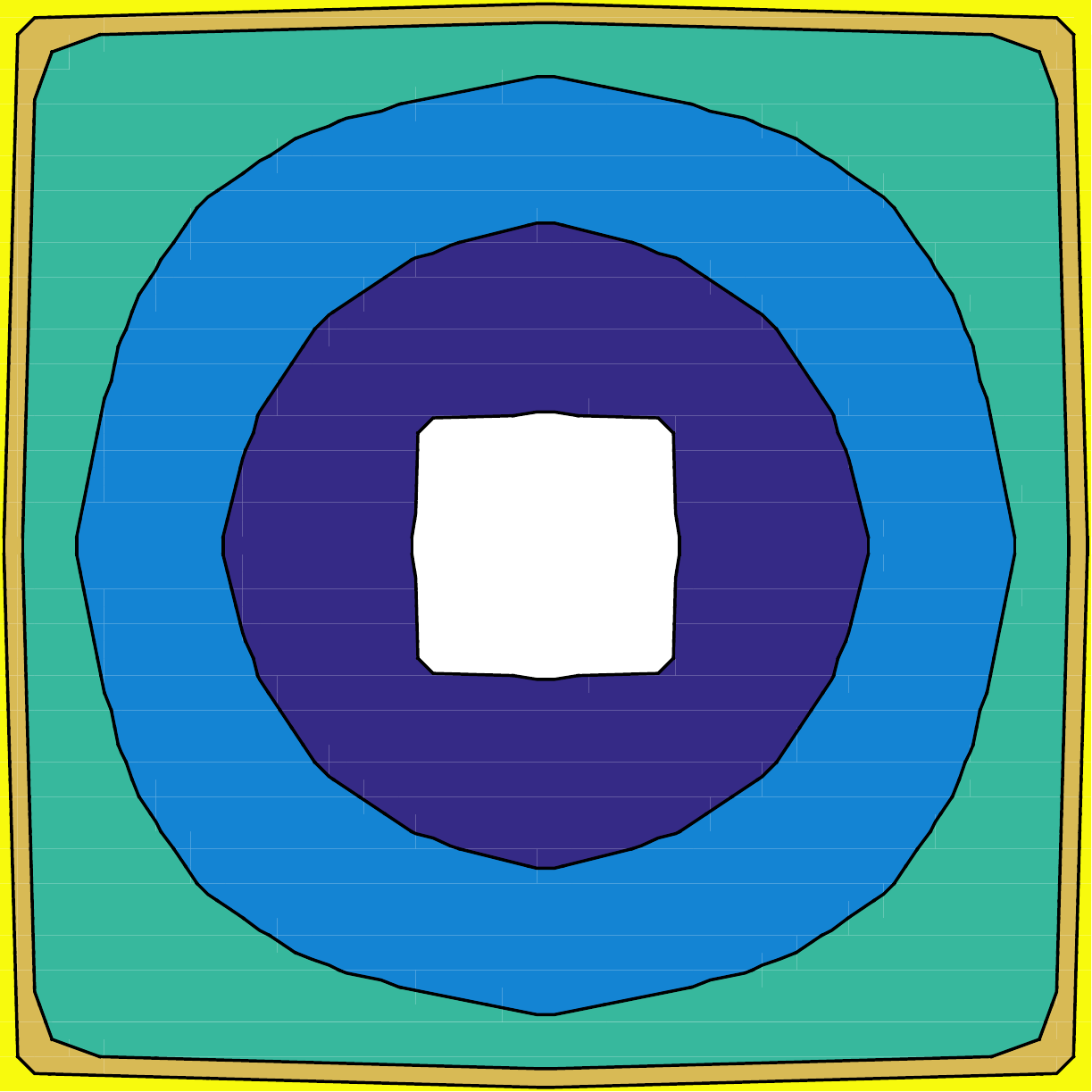}
%\hfill\hspace{2cm}
\includegraphics[width=.3\textwidth]{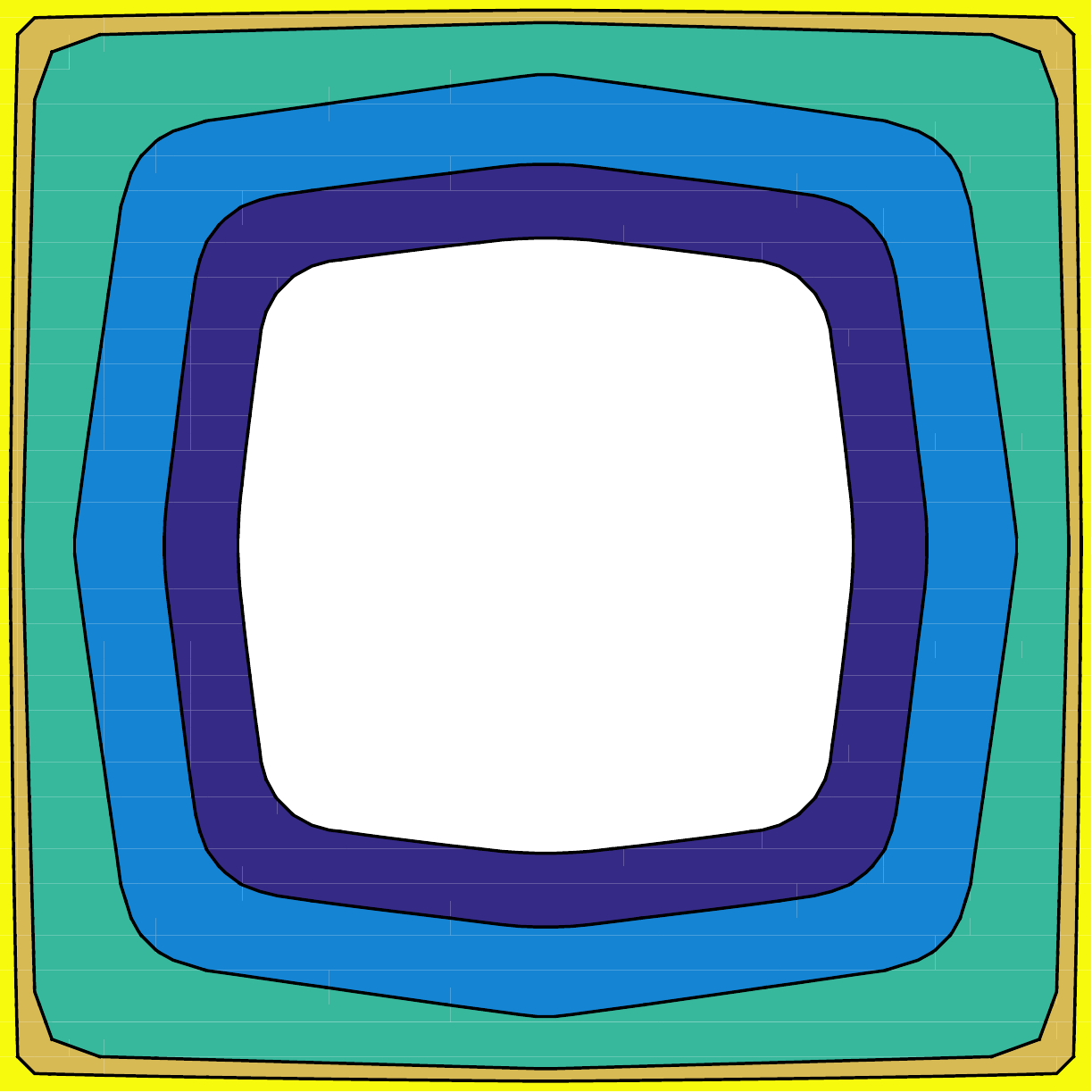}
\caption{Example \ref{example_robust}. Top: Surface plots (inverted) of function, robust QCE,  strictly QCE.. Bottom: Corresponding contour plots.}
\label{fig:robust}
\end{figure}

\end{example}
%%%%%%%%%%%%%%%%%%%%%%%%%%%%%%%%%%%%%%%%%%%%%%%%%%%%%%%%%
\subsection{Accelerating iterations using the line solver}
We found an effective method to reduce the computational time to find the solution. We implement the line solver for the robust QCE proposed in \cite{AbbasiLineSolverQuasiConvex},  alternating with the PDE iterations. In particular, on an $n\times n$ grid, after every $2n$ iterations of the PDE we apply one iteration of the line solver (with the commensurate value of $\e$) in each direction. \autoref{table_accelerate} shows the number of iterations required for convergence using the same obstacle from \cite[Example 6.1]{AbbasiLineSolverQuasiConvex}.   Note the line solver is for a different regularization of the QC operator, however for small values of $\e$ the solutions are close.  For larger values of $\e$ the operator still accelerates the solver, but it does not approach the solution to within arbitrary precision.

%Although the line solver is faster than the mixed method, the PDE has the added benefit of regularity, as well as having strictly convex level sets. 

In table~\ref{table_accelerate} we compare the number of iterations required for each method.  The results were comparable across different examples.  The number of iterations of the PDE operator required for convergence was typically a small constant times $n^2$.
The computational cost of a single line solve was on the order of a small constant (say 10) times a the cost of a single PDE iteration. So the combined method, requires $\bO(n)$ iterations (possibly with a $\log n$ prefactor, which we ignore, since it is not significant at these values of $n$).   So the combined method required (roughly)  $n$ iterations. 

 Recall that there are $N =n^2$ grid points. Each iteration has a cost proportional to the number of directions used and to the number of grid points, so the total cost is $\bO\left((WN^2) \right)$ for the iterative method, and $\bO(WN)$ for the combined method (with constants of roughly 1 and 10, respectively).   So combining the iterative solver with the line solver results in significant improvements to computation time.

\begin{table}[t]
\centering
\begin{tabular}{c | c | c | c | c }
$n$ & 32 &   64 &  128  &  256 \\
\hline
2n PDE Iterations & 50 & 102 & 203 & 398 \\
2n PDE It + 1 Line Solve & 6 & 8 & 10 & 7
\end{tabular}
\caption{Number of iterations required for convergence using different solution methods.}
\label{table_accelerate}
\end{table}

\bibliographystyle{alpha}
\bibliography{QC_PDE_main}

\end{document}